\documentclass{article}     

\usepackage{amsmath}
\usepackage{amsthm}
\usepackage{amssymb}
\usepackage{latexsym}
\usepackage{graphics}
\usepackage{euscript}
\usepackage{anysize,amssymb}
\usepackage{enumerate}
\usepackage{tikz}
\usetikzlibrary{positioning,arrows,fit}
\usetikzlibrary{decorations.pathmorphing}

\def\A{{\mathfrak A}}

\def\kn{\kern.1em}

\def\Mod{{\rm Mod}}
\def\Log{{\rm Log}}
\def\liminv{{\rm lim }_{\leftarrow}}
\def\Fmc{\mathcal{F}}
\def\Cla{\Mod}
\def\C{\mathcal{C}}
\def\K{{\mathsf K}}
\def\inft{\infty}
\def\ue{{u.e.}}
\def\b{\Box}
\def\d{\Diamond}
\def\Lf{{{\cal L}f}}
\def\iffa{\Longleftrightarrow}
\def\rank{\mbox{ rank\,}}
\def\P{\mathcal{P}}
\def\a{\mathfrak{a}}

\def\F{\mathfrak{F}}
\def\D{\mathfrak{D}}
\def\T{\mathfrak{T}}

\def\V{\mathbf{V}}
\def\u{\mathbf{u}}
\def\v{\mathbf{v}}
\def\vv{\vec{\v}}
\def\vu{\vec{\u}}

\def\Del{\mathsf{Del}}
\def\Arr{\mathsf{Arr}}

\def\rank{\mathsf{Rank}}
\def\dist{\mathsf{Dist}}

\def\pr{{pr}}
\def\e{\mathsf{e}}
\def\E{\e}
\def\k{\mathsf{k}}
\def\A{\mathsf{a}}
\def\Tt{{\tilde \T}}

\renewcommand{\theequation}{\Roman{equation}}

\theoremstyle{definition}
\newtheorem{teo}{Theorem}[section]  
\newtheorem{cor}[teo]{Corollary}
\newtheorem{prop}[teo]{Proposition}
\newtheorem{defi}[teo]{Definition}
\newtheorem{ex}[teo]{Example}
\newtheorem{lemma}[teo]{Lemma}

\author{Stanislav Kikot}
\title{A dichotomy for some elementarily generated modal logics}
\begin{document}
\maketitle

\begin{abstract}
In this paper we consider the normal modal logics of elementary classes defined
by first-order formulas of the form $\forall x_0 \exists x_1 \dots \exists x_n \bigwedge x_i R_\lambda x_j$.
We prove that many properties of these logics, such as finite axiomatisability, elementarity, 
axiomatisability by a set of canonical formulas or by a single generalised Sahlqvist formula, together with modal definability of the initial formula, either simultaneously hold or simultaneously do 
not hold.
\end{abstract}

\section{Introduction}

This research was motivated by the following observation. Consider two first-order conditions: 
$\forall x \exists y (x R y \land y R x)$ and $\forall x \exists y (x R y \land y R y)$ 
(see Figure~\ref{fig:1}).
The first one is modally definable by a Sahlqvist  formula $p \to \d \d p$ while the second is not, since it
does not reflect ultrafilter extensions (e.g., \cite{Bl}, p. 142). 
The difference between these two formulas becomes even more palpable
if we look at the modal logics $L_1$ and $L_2$ of the corresponding elementary classes. While $L_1$ is
axiomatisable (with the standard rules of Substitution, Modus Ponens and Necessitation) 
by a single Sahlqvist formula, $L_2$ is not finitely
axiomatisable and the class of Kripke frames $\{\F \mid \F \models L_2\}$ is not definable by any formula of first-order logic \cite{hughes}. Moreover, any axiomatisation of $L_2$ requires infinitely many non-canonical formulas
\cite{Hodkinson03canonicalvarieties}. On the other hand, both formulas have a common structure and 
can be represented by graphs as in Figure~\ref{fig:1}, which are called diagrams in this paper.
\begin{figure}[t]
\centering
\scalebox{0.88}{\hbox{\includegraphics{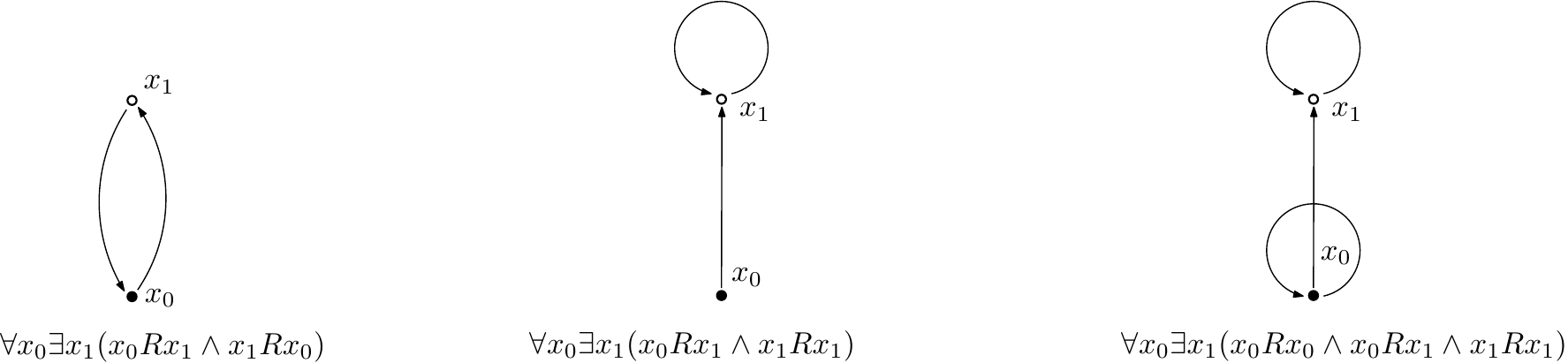}}}
\caption{Formulas and their diagrams (universally quantified variables are black, existentially quantified variables are white)}
\label{fig:1}
\end{figure}
The author decided that this issue is worthy of additional explanation.
So  a study with the purpose of classifying all elementary classes $\C$ 
definable by formulas of the form 
$\forall x_0 \E(x_0)$ where $\E(x_0) =  \exists x_1 \dots \exists x_n \bigwedge x_i R_\lambda x_j$ 
according to the following model-theoretic properties (whose precise definitions will be given in Section 2) was undertaken: 
\begin{enumerate}[(\it {I}-i)]
\item $\E(x_0)$ is modally definable by a generalised Sahlqvist formula; 
\item $\E(x_0)$ is locally modally definable; 
\item $\forall x_0 \E(x_0)$ is globally modally definable; 
\item $\Log(\C)$ is axiomatisable by a generalised Sahlqvist formula;
\item $\Log(\C)$ is finitely axiomatisable;
\item $\Log(\C)$ is axiomatisable by a set of modal formulas containing finitely many propositional variables;
\item $\Log(\C)$ is axiomatisable by a set of canonical formulas;
\item $\Log(\C)$ is axiomatisable by a modal formula $\phi$ and a set of canonical formulas;
\item $\{\F \mid \F \models \Log(\C)\} = \C$;
\item $\{\F \mid \F \models \Log(\C)\}$ is elementary\footnote{In this paper we call a class of first-order models elementary if it is defined by a single first-order sentence, and $\Delta$-elementary
if it is defined by a set of sentences.}.
\end{enumerate}

Briefly, we prove that for any class $\C$ in question, conditions ({\it I-i}) -- ({\it I-x}) either simultaneously hold, or simultaneously do not hold, and this is determined by the existence in the corresponding 
diagram of an undirected cycle not passing through the universally quantified point, provided that the diagram is 
``minimal'', i.e., none of its edges may be removed without affecting the corresponding formula, and ``rooted'', i.e., each of its points is reachable from $x_0$ via a directed path.

We exclude from our list such algorithmical properties as decidability, finite model property and complexity, and do not deal with them in this paper, since an easy (but seemingly unpublished) argument shows that all logics in our class have f.m.p. and are PSPACE-complete regardless of the mentioned cycle. But we cannot help mentioning that the dichotomies 
in the
complexity-theoretic setting have recently  become known to the logical community. 
For example, in \cite{hemaspaandra} the modal logics given by universal Horn sentences are classified into those that are in NP and 
those that are PSPACE-hard and this classification was further refined in \cite{DBLP:conf/lics/MichaliszynO12}. 
The authors of  \cite{DBLP:conf/aiml/KuruczWZ10} classified universal relational constraints with respect to 
the complexity of reasoning in the description logic $\mathcal{EL}$.

This work is in line with current research in theoretical modal logic. 
First, this result can be considered as a straighforward generalisation of Hughes' paper \cite{hughes} about the reflexive-successor logic.  The axiomatics of \cite{hughes} was generalised in \cite{BSS} to the case of first-order conditions of the form 
$\forall x \exists y (xR_\lambda y \land \phi(y))$ where $\phi(y)$ is a generalised Kracht formula
\cite{K2}, 
and for some particular logics of this form  finite axiomatisability, the finite model property and elementarity are studied there. 
The authors of \cite{BSS} also conjectured that within their class there is a coincidence between 
finite axiomatisability and elementarity, and between $\Delta$-elementarity and elementarity 
(cf. \cite{Benthem76a}).

Another central problem of modal logic is: given an elementary class, 
i.e., a first-order formula, provide an explicit axiomatisation of the corresponding modal logic (this was done in \cite{Hodkinson}), and describe its properties, for example, in terms of ({\it I-i})--({\it I-x}) (cf. problems 6.6 and 6.8 ibid.) Since the product of two elementary classes is elementary \cite{GSH1}, 
the school of many dimensional modal logic deals mainly with such problems 
(e.g., \cite{Agi2},\cite{DBLP:conf/aiml/Kurucz10}, and \cite{GZ} for older results). In general, the algorithmic problem `given a first-order formula, decide whether each of ({\it I-i})--({\it I-x}) holds' 
should be undecidable due to the undecidability of first-order logic. E.g., for ({\it I-iii}) it is Chagrova's theorem
\cite{DBLP:journals/jsyml/Chagrova91},\cite{ChL}, but it seems plausible that using the method of 
\cite{ChL} one can prove such undecidability results for all items. On the other hand, when we restrict attention
to a fragment of the first-order language with decidable implication, then we have chances to obtain such algorithmic criteria (as, e.g. in \cite{KHorn}), and the present paper is a step in this direction.

One more fundamental problem of modal logic is to study which implications between
({\it I-i})--({\it I-x}) hold, and which of these conditions are independent. A brief summary of known results
is given in \cite{Hodkinson} (see discussion after Problem 6.6), and we think that our result is interesting in
this context.

This paper also concerns a phenomenon called ``canonicity in the limit'', referring to the logics 
(or, more generally, in terms of universal algebra,  sets of equations that are true on some elementary class), 
that are canonical, but cannot be axiomatized by canonical formulas, and, even more, any axiomatisation of such logics requires infinitely many non-canonical axioms. Beside the aforementioned Hughes' logic, this issue includes the equational theories of representable relational \cite{Hodkinson03canonicalvarieties} 
and cylindrical \cite{Bulian} algebras,  and the well-known McKinsey-Lemmon logic \cite{GoldblattHodkinson2006}. 
This issuee was elaborated further in \cite{BulianHodkinson}.
It turns out that all logics under consideration in the present paper excepting those which are generally Sahlqvist have this property; thus, what was thought pathological for elementarily generated modal logics
can now be seen to be the norm.

And --- last but not least --- this paper can be regarded as a contribution to the question of whether there are natural generalisations of the Sahlqvist-Kracht correspondence in the basic modal language, besides \cite{Gor2}, \cite{V}, \cite{V2}, \cite{K2}. From our result it follows that the Kracht's theorem \cite{Kr1}, \cite{Kr2} cannot be generalised further within the formulas of our class with any of ({\it I-i})--({\it I-x}) as its consequence.

The outline of the paper is as follows. First, we take  a diagram, 
all the cycles of which pass through the root, and use the result from \cite{KZ} stating that 
the corresponding first-order formula 
$\E(x_0)$ is modally definable by a generalised Sahlqvist formula, and so, 
by the generalised Sahlqvist theorem \cite{Gor2},  
({\it I-i}) -- ({\it I-x}) hold. Then we have to 
take a diagram with a cycle not passing through the root, and show that ({\it I-i}) -- ({\it I-x}) do not hold. 
This can be done only if $\E(x_0)$ is ``minimal'', i.e., it does not contain atoms which can be thrown away without changing $\E(x_0)$ semantically. Indeed, the diagram in Figure~\ref{fig:1} on the right has a cycle not passing through the root, but it is modally definable, since it is equivalent to the reflexivity condition. So we additionally assume that the diagram is minimal. This can be done without any loss of generality, since we may take any formula of our class and remove superfluous edges until the formula becomes minimal. Under this assumption we prove that 
({\it I-i}) -- ({\it I-x}) do not hold in Sections 4 -- 8. For this purpose we need the axiomatisations of the corresponding modal logics,  
constructed in Section \ref{section:axiomatization}. Then we construct `non-standard frames' for our logic ensuring falsity of  ({\it I-i}) -- ({\it I-x}). 

\begin{figure}[t]
\centering
\includegraphics{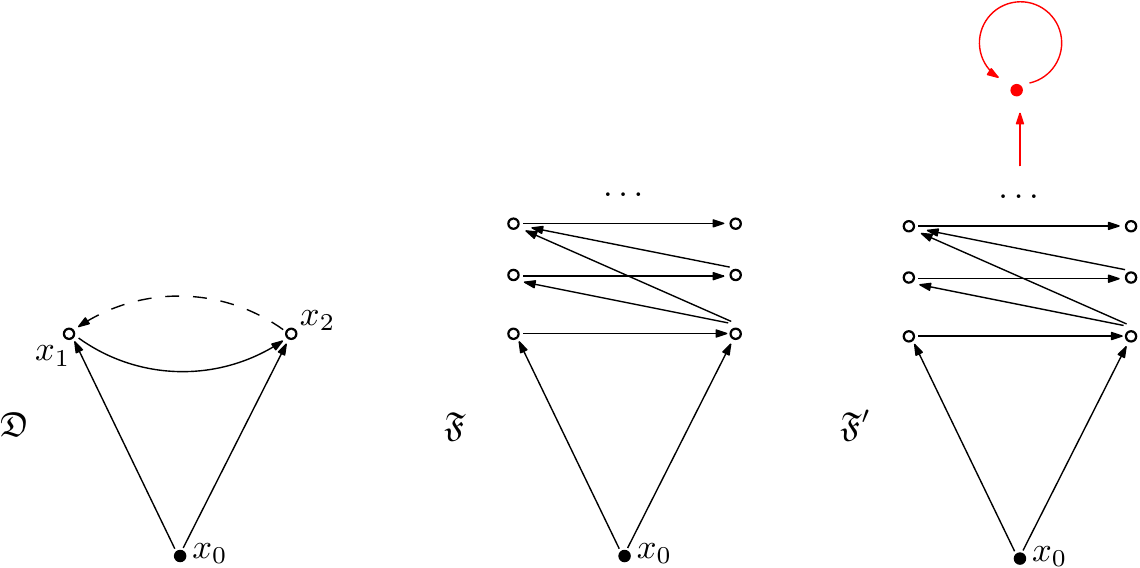}
\caption{}
\label{fig:outline}
\end{figure}

To understand the intuition underlying these non-standard frames and the problems arising in their construction, let us turn to \cite{KZ}, where similar frames are used to prove that ({\it I-ii}) does not hold for any diagram 
with a cycle of the given form.
Roughly, we temporarily remove one of the edges of a diagram $\D$ belonging to a cycle (dashed edge in $\D$ in Figure~\ref{fig:outline}), then clone the rest of the diagram 
(more precisely, all points except the root) $\omega$ times,  preserving edges of the diagram inside each layer and between the root and any layer, and insert the deleted arrow between corresponding points of different layers, from layers with lesser numbers to layers with greater numbers. Denote this Kripke frame by $\F$ (Figure~\ref{fig:outline}, in the middle). 
One can show that $\F$  has a root $r$ and satisfies $\F \not \models \E^\D(r)$ and $(\F)^{u.e.} \models \E^\D(r)$, yielding the negation of ({\it I-ii}). Now imagine that we want to generalise this construction to prove that  ({\it I-iii}) does not hold. In this case we need our construction to validate a stronger condition
$(\F)^{u.e.} \models \forall x \E^\D(x)$, so the construction must be modified accordingly. 
In many cases we can just `put on top' of  $\F$ a reflexive point (Figure~\ref{fig:outline}, on the right), 
but in general this approach does not work and
a more subtle construction is required. So in Lemma \ref{rank1} we `saturate' $\D$ by adding new points until it satisfies
$\forall x \E^\D(x)$, and thus construct a Kripke frame $\F^\D_+$. 
Then we delete an edge of the cycle, obtaining $\F^\D_-$, and use $\F^\D_+$ and $\F^\D_-$ 
instead of $\D$ in the construction of $\F$. Another component of these non-standard frames are probabilistic graphs of I. Hodkinson and Y. Venema.

\section{Preliminaries}

\subsection{Modal Formulas and Logics}

Fix a set of \emph{propositional variables} $PV=\{p_1, p_2, \dots\}$, a set of \emph{nominals} $NV = \{j_1, j_2, \dots\}$, and  a set of indices $\Lambda$. Propositional variables are also denoted by $p, q, r,\dots $ and nominals 
(only in this section) by $i$ and  $j$. 
\emph{Hybrid formulas} are built from propositional variables and nominals using
the constant $\bot$, the binary connective $\to$, and unary connectives $\d_\lambda$ for $\lambda \in \Lambda$ 
and $\exists i$ for $i \in NV$. 
Other constructs are defined as usual: in particular, $\b_\lambda$ is a shorthand for $\neg \d_\lambda \neg$ and
$\forall i $ is a shorthand for $\neg \exists i  \neg$. 
A \emph{Kripke frame} is a tuple $\F=(W^\F, (R^\F_\lambda: \lambda \in \Lambda))$ 
where $W^\F$ is a set (the carrier) and the $R^\F_\lambda$ are binary relations on $W^\F$;
instead of $(x,y) \in R^\F_\lambda$ we often write $x R^\F_\lambda y$ or $\F \models x R_\lambda y$. 
We consider the components of $\F$ as sets rather than the parts of the interpretation mapping, so occasionly
their superscripts may be different from $\F$ or just omitted.
A map $\theta: PV \to \P(W^\F)$ is called a \emph{propositional 
valuation (for a frame $\F$)}, and a map $\tau: NV \to W^\F$ is called a \emph{nominal valuation}.
A propositional valuation $\theta$ in a Kripke frame is called $k$-generated, 
if there are at most $k$ different propositional variables $p$, such that
$\theta(p)\neq\emptyset$.
Given a Kripke frame $\F$ and  valuations $\theta$ and $\tau$, 
we define the \emph{truth} of hybrid formulas in a point $x \in W^\F$ as usual:
$$\F, \theta, \tau, x \models p \iffa x \in \theta(p); \qquad \F, \theta, \tau, x \models j \iffa x = \tau(j);
\qquad \F, \theta, \tau, x \models \bot \mbox{ never };$$
$$\F, \theta, \tau, x \models \d_\lambda \phi \iffa \exists y \in W^\F \mbox{ such that } x R^\F_\lambda y \mbox{ and } 
\F, \theta, \tau, y \models \phi;$$
$$\F, \theta, \tau, x \models \phi \to \psi \iffa \mbox{ either }
\F, \theta, \tau, x \not\models \phi \mbox{ or } \F, \theta, \tau, x \models \psi;$$
$$\F, \theta, \tau, x \models \exists i\, \phi \iffa 
\begin{array}{c}
\mbox{for some nominal valuation $\tau'$ }
\mbox{ such that }\\ \tau'(j) = \tau(j) \mbox{ for all $j \in NV \setminus \{ i \}$ } 
\mbox{we have } \F, \theta, \tau', x \models \phi.
\end{array}
$$

A \emph{modal formula} is a hybrid formula without nominals and nominal quantifiers. The truth of a modal formula
$\phi$ at a point $w\in W^\F$ of a Kripke frame $\F$ depends only on the propositional valuation $\theta$ and is denoted by $\F, \theta, w \models \phi$. A modal formula $\phi$ is said to be \emph{valid in a point} $w \in W^\F$ of a Kripke frame $\F$ (denoted by $\F, w \models \phi$) if for all valuations $\theta$ we have $\F, \theta, w \models \phi$, and
is said to be \emph{valid in a Kripke frame $\F$} (denoted $\F \models \phi$) if for all $w\in W^\F$ we have
$\F, w \models \phi$.

We can regard Kripke frames as models for the classical first-order language $\Lf_\Lambda$, the signature of which consists of binary predicate symbols $R_\lambda$ for all $\lambda\in \Lambda$. We assume that $\Lf_\Lambda$ contains equality. 
The formulas of this language 
are called in the sequel simply \emph{first-order formulas.} 
Thus the truth relation $\F \models \A$ is also defined for closed first-order formulas $\A$,
also called \emph{first-order sentences}.
A first-order sentence $\A$ is said to be globally modally definable if there exists a modal formula $\phi$
such that for any Kripke frame $\F$, $\F \models \phi$ iff $\F \models \A$. A first-order formula $\A(x)$ with a single
free variable $x$ is said to be locally modally definable if for any Kripke frame $\F$ and any point $w$ in $\F$,
$\F, w \models \phi$ iff $\F \models \A(w)$ for some modal formula $\phi$.

Each first-order sentence $\A$ defines the class of Kripke frames $\Cla(\A) = \{\F \mid \F \models \A\}$. A class of Kripke frames $\C$ is said to be elementary if it is of this form, and $\Delta$-elementary if 
$\C = \bigcap_{k=1}^\infty \Cla(\A_k)$ for some sequence of first-order formulas $\{\A_k\}$.
Given a class of Kripke frames $\C$, by $\Log(\C)$ we denote 
the set of all modal formulas valid in all frames of $\C$.
A normal modal logic is a set of modal formulas containing all propositional tautologies, the formulas
$$\K_\lambda: \b_\lambda(p\to q)\to(\b_\lambda p \to \b_\lambda q), \quad \lambda\in\Lambda,$$  
and closed under inference rules Modus Ponens, Uniform Substitution and Necessitation:
$$
\begin{array}{c}
\phi, \phi\to \psi\\
\hline
\psi\\ 
\end{array} 
,
\quad\quad
\begin{array}{c}
\phi\\
\hline
\phi[\psi/p] \\ 
\end{array} 
,
\quad\quad
\begin{array}{c}
\phi\\
\hline
\b_\lambda \phi \\ 
\end{array} 
.
$$
It is easy to see that $\Log(\C)$ is always a normal modal logic. 
A set of modal formulas $\Sigma$ axiomatises a normal modal logic $L$ if $L$ is the minimal 
(w.r.t. set-theoretic inclusion) normal modal logic containing $\Sigma$, and in this case we write
$L = \K + \Sigma$.  A modal logic is said to be
finitely axiomatisable if it is axiomatised by some finite $\Sigma$, and axiomatisable using finitely
many variables if it is axiomatised by some $\Sigma$ such that only a finite number of propositional 
variables occur in $\Sigma$. A modal formula $\phi$ is said to be \emph{canonical} if it is valid in the 
canonical frame of the normal modal logic axiomatised by $\phi$. For a set $B$ we denote its powerset by $\P(B)$.

\subsection{Minimal Diagrams and Kripke frames}

Consider an arbitrary relational structure 
$\F = (W^\F, (R^\F_\lambda:\lambda\in\Lambda))$ where $R^\F_\lambda$ are binary relations on $W^\F$. 
For a binary relation $R$ by $\breve{R}$ we denote its converse $\{(x,y) \mid (y,x) \in R\}$.
A sequence $z_1\lambda_1 z_2\lambda_2\ldots \lambda_h z_{h+1}$ where for all $i$ 
$z_i \in W^\F$, $\lambda_i \in \Lambda$ and $(z_i, z_{i+1}) \in R^\F_{\lambda_i}$ is called
a {\sl directed path of length $h$ 
connecting $z_1$ to $z_{h+1}$ in $\F$}. To obtain the definition of an {\sl undirected path},
we put $\Lambda^\pm = \Lambda \cup \{ \lambda^- \mid \lambda \in \Lambda\}$, 
assume that $R^\F_{\lambda^-} = \breve{R^\F_\lambda}$
and replace $\Lambda$ with $\Lambda^\pm$ in the definition of a directed path. 
An {\sl undirected cycle} is an undirected path of positive length $h$ with $z_1=z_{h+1}$ and
not containing a subsequence of the form $z \lambda y \lambda^- z$ and $z \lambda^- y \lambda z$.  
The {\sl distance } from $y \in W^\F$  and $z\in W^\F$ in $\F$ (denoted by $\dist_\F(y,z)$) is the length of the shortest directed path connecting $y$ to $z$ in $\F$; if such path does not exist, we put $\dist_\F(y,z) = \infty$.

In this paper we deal with \emph{pointed Kripke frames}, in other words, with tuples of the form 
$\D = (W^\D, (R^\D_\lambda: \lambda \in\Lambda), x^\D_0)$, 
where $R^\D_\lambda$ are binary relations on $W^\D$  and $x^\D_0 \in W^\D$. 
A pointed Kripke frame $\D$ is called \emph{rooted} if for any point $y$ of $W^\D$ there exists a directed path leading from $x^\D_0$ to $y$. A \emph{diagram} is a pointed rooted Kripke frame with a finite domain. 
 An undirected cycle in a diagram $\D$  is said to be {\sl inner} if it does not contain the root $x^\D_0$.

Assume that $W^\D = \{x_0, x_1, \ldots, x_n\}$ and that $x^\D_0 = x_0$. The symbols $x_i$ will play a double role:
first, they are the points of the diagram, second, they are first-order variables in the formulas defined below.
We set
$$\k^\D(x_0, x_1, \ldots, x_n) = \bigwedge\{x_i R_\lambda x_j \mid i,j \le n,\, x_i R^\D_\lambda x_j\}$$ 
and
$$\e^\D(x_0) = \exists x_1 \ldots \exists x_n \k^\D ( x_0, x_1, \ldots, x_n).$$

We say that a diagram $\D' = (W^{\D'}, (R^{\D'}_\lambda: \lambda \in\Lambda), x^{\D'}_0)$ \emph{is obtained from 
a diagram $\D = (W^\D, (R^\D_\lambda: \lambda \in\Lambda), x^\D_0)$ by deleting the edge } $x R_{\lambda_0} y$ 
and write $\D' = \D - (x, y, \lambda_0)$ if $(x,y) \in R^\D_{\lambda_0}$,
$W^{\D'}= W^\D$, $x^{\D'}_0 = x^\D_0$,
 $R^{\D'}_{\lambda_0}= R^\D_{\lambda_0} \setminus\{(x,y)\}$ and for all $\lambda \neq \lambda_0$ 
$R^{\D'}_\lambda= R^\D_\lambda$.
A diagram  $\D$ is called \emph{globally (locally) minimal} if for any diagram $\D'$ obtained from  $\D$ by deleting an edge it is not true that $\vdash_{FOL} \forall x_0 \E^{\D'}(x_0) \to \forall x_0 \E^\D(x_0)$ 
(respectively, $\vdash_{FOL} \forall x_0(\E^{\D'}(x_0) \to \E^\D(x_0))$); $FOL$ here means the classical first-order logic. Global minimality implies local minimality, 
but  the converse fails in general.  For instance, the diagram $\D$ corresponding to the formula $\E^\D(x_0) = \exists x_1 \exists x_2 (x_0 R x_1 \land x_1 R x_2)$ is locally but not globally minimal.

{
\def\Iff{\Longleftrightarrow}
\def\IFF{\Longleftrightarrow}
\subsection{Ultrafilters, ultrafilter extension and ultraproducts}\label{Sect:Ultrafilter}

A set $u \subseteq \P(W)$ is an
\emph{ultrafilter} over a set~$W$ if, for all ${X,Y \subseteq W}$,
\begin{itemize}
\item[(u1)] if $X,Y \in u$, then  $X\cap Y \in u$;
\item[(u2)] if $X \in u$ and $X \subset Y$, then  $Y \in u$;
\item[(u3)] $X \notin u$ iff  $\bar X \in u$, where $\bar X$ denote the complement $W\setminus X$.
\end{itemize}

From the definition it follows that for any ultrafilter $u$ over a set $W$
$\emptyset \notin u$ and $W \in u$.

Given a frame $\F = (W, (R_\lambda: \lambda\in\Lambda))$,
its \emph{ultrafilter extension} is defined as 
the frame
$\F^\ue = (W^\ue,(R_\lambda^\ue :\lambda\in\Lambda))$, where
$W^\ue$ is the set of all ultrafilters over~$W$,
and $u R_\lambda^\ue u'$ holds for ultrafilters~$u$ and~$u'$
iff $R_\lambda^{-1}(X)\in u$ for all ${X \in u'}$, where 
$R_\lambda^{-1}(X) = \{ z \mid z R_\lambda x \mbox{ for some } x\in X\}$.
Given a point $a\in W$, the set $\pi_a=\{X\subseteq W\mid a \in X\}$
is obviously an ultrafilter; it is called the \emph{principal ultrafilter} generated by~$a$.

\begin{lemma}[\cite{Bl}, p. 95]\label{uf3}
For any points $a,b$ in any frame $\F$, 
$a R_\lambda b$ $\Iff$ $\pi_a R_\lambda^\ue\pi_b$.
\end{lemma}

{
\def\Fi{\phi}

\begin{lemma}[\cite{Bl}, p.~142]\label{Lemma:Standard:UE}
For any frame $\F\!$ and  modal formula~$\Fi$, 
$\F^\ue \models\Fi$ \ implies $\F\models\Fi$.
\end{lemma}
}

\begin{lemma}\label{ultra-n}
Let $u$ be an ultrafilter over $W$,  $W = W_1 \cup W_2 \cup \dots \cup W_n$, 
and $W_i \cap W_j = \emptyset$ for all $1 \le i\neq j \le n$. Then there exists
a unique $i$ such that $W_i \in u$.
\end{lemma}
\begin{proof}
Induction on $n$. The cases $n=1$, $n=2$ follow immediately from the definition of an ultrafilter.
Now suppose that the lemma is proven for some fixed $n$ and let us prove that it is
true for $n+1$.  Let $V_1 = W_1, \dots, V_{n-1} = W_{n-1}, V_n = W_{n}\cup W_{n+1}$. By inductive
assumption we get either $W_i \in u$ for some $1 \le i \le n-1$, or $W_{n}\cup W_{n+1} \in u$.
In the last case, if $W_{n} \notin u$ and $W_{n+1} \notin u$, then by (u3) we have
$W_1 \cup \dots \cup W_{n-1} \cup W_{n+1} \in u$ and $W_1 \cup \dots \cup W_{n}  \in u$, therefore,
by (u2) $W_1 \cup \dots \cup W_{n-1} \in u.$ This  contradicts (u3) and $W_{n}\cup W_{n+1} \in u$.
Thus there exists $1\le i \le n+1$ such that $W_i \in u$. If there are two such $i$'s, then
$\emptyset \in u$.
\end{proof}

We need yet another model-theoretic construction involving ultrafilters.
Suppose that we have a family of  Kripke frames 
$\F^i=(W^i, (R^i_\lambda:\lambda\in\Lambda) )$ for all $i \in \omega$
and a non-principal ultrafilter $u$ over $\omega$.
We say that two sequences 
$\bar \alpha = (\alpha_0, \alpha_1, \alpha_2, \ldots)$
and $\bar \beta = (\beta_0, \beta_1, \beta_2, \ldots)$, where $\alpha_i, \beta_i \in W^i$ for all 
$i\in \omega$ are {\sl $u$-equivalent} (denoted by  $\bar \alpha \sim_u\bar\beta$), 
if $\{i\mid \alpha_i = \beta_i\}\in u$. By $W$ we denote the set of all such sequences.
The equivalence class of a sequence
$\alpha$ we denote  by $\lceil \alpha \rceil$.
The $\Lf_\Lambda$-structure 
$\F=(W', (R'_\lambda:\lambda\in\Lambda) )$,  where 
$$W'=\{ \mbox{ all sequences of points from } W\} / \sim_u,$$
and $\lceil \bar\alpha \rceil R'_\lambda \lceil \bar \beta\rceil
\iffa \{i\mid \alpha_i R^i_\lambda \beta_i \} \in u$ 
is called an {\sl ultraproduct } of $\F^i$ and is denoted by
$\prod^u_{i\in\omega} \F^i$.

\begin{prop}[\cite{KeislerChang}, Thm 4.1.9]
\label{ultraproduct}
If $\C$ is an elementary class and $\{\F^i\}$ is a sequence of Kripke frames  from $\C$, then
for any ultrafilter $u$ on $\omega$, $\prod^u_{i\in\omega} \F^i \in \C$.
\end{prop}
}

\subsection{Inverse limit of descriptive frames} 

A \emph {general frame} is a triple $(W, (R_\lambda: \lambda \in \Lambda), P)$ where 
$(W, (R_\lambda: \lambda \in \Lambda))$ is a Kripke frame and $P \subseteq \P(W)$ is non-empty and
closed under intersection, complement and $R_\lambda^{-1}$. A general frame
$(W, (R_\lambda: \lambda \in \Lambda), P)$ is said to be a \emph{descriptive frame} if
\begin{enumerate}
\item If $x,y \in W$ are distinct, then there is some $S \in P$ with $x\in S$ and $y \notin S$.
\item If $x, y \in W$ and $\neg R_\lambda(x,y)$, then there is some $S \in P$ with $x \in 
R_\lambda^{-1}(S)$ and $y \notin S$.
\item $\bigcap\mu \neq 0$ for every $\mu \subseteq P$ with finite intersection property.
\end{enumerate}
Below we denote general frames by calligraphic letters to distinguish between them and Kripke frames.
If $\F = (W, (R_\lambda: \lambda \in \Lambda))$ is a Kripke frame, we write $\F^+$ for 
$(W, (R_\lambda: \lambda \in \Lambda), \P(W))$. Clearly, if $\F$ is finite (i.e., $W$ is finite),
then $\F^+$ is a descriptive frame. If ${\cal F}=(W, (R_\lambda:\lambda \in \Lambda), P)$ is a descriptive
frame, we write ${\cal F}_+$ for its underlying Kripke frame $\F=(W, (R_\lambda:\lambda \in \Lambda))$.
Let $\mathcal{F}=(W, (R_\lambda:\lambda \in \Lambda), P)$ be a general frame and $\phi$ a modal formula. We say that $\phi$ is valid in $\mathcal{F}$, written $\mathcal{F}\models \phi$ if $(W, (R_\lambda:\lambda \in \Lambda), 
\theta, w \models \phi$ for every assigment $\theta: PV \to P$ and every $w \in W$.

An inverse family of descriptive frames is an object $\mathcal{I} = ((I, \le), \mathcal{F}_i, (f_{ij}: i \ge j 
\mbox{ in } I))$ where $(I, \le)$ is an upwards-directed partial order ('upwards-directed' means that any finite subset of $I$ has an upper bound in $I$), $\mathcal{F}_i = (W_i, ((R_i)_\lambda : \lambda \in \Lambda), P_i)$ 
is a descriptive frame for each $i \in I$, and for each 
$i, j \in I$ with $i \ge j$ $f_{ij}: \mathcal{F}_i \to \mathcal{F}_{j}$ is a frame homomorphism such that (a) $f_{ii}$ is the identity map on $W_i$, and (b) $f_{jk} \circ f_{ij} = f_{ik}$ whenever $k \le j \le i$ in $I$.

The inverse limit $\liminv \mathcal{I}$ of $\mathcal{I}$ is defined to be 
$\mathcal{F} = (W, (R_\lambda: \lambda \in \Lambda), P)$ where
$$W = \{ x \in \prod_{i \in I} W_i : f_{ij}(x_i) = x_j \mbox{ for each } i \ge j \mbox{ in } I\},$$
$$R_\lambda = \{ (x,y) \in W :  x_i (R_i)_\lambda y_i  \mbox{ for each } i \in I \},$$
$$P \mbox{ is generated by } \{f_i^{-1}[S]: i \in I, S \in P_i\},$$
where in the last line for each $i \in I$ $f_i: W \to W_i$ is the projection given by $f_i(x) = x_i$. 

\begin{prop}[\cite{G8}, 1.1.2(8), 1.11.4]\label{fact} 
The inverse limit $\mathcal{F}$ of $\mathcal{I}$ is itself a descriptive frame. Moreover,
for any modal formula $\phi$, if $\phi$ is valid in $F_i$ for each $i$, then $\phi$ is valid in $\mathcal{F}$.
\end{prop}
Ignoring the line defining $P$, we obtain the definition of an inverse limit 
$\liminv \mathcal{I}$ of  families of Kripke frames and graphs. If $I = (\omega, \le)$ and $f_{ij}$ are
clear from context, we denote the system simply $\{F_i\}$, and the inverse limit by $\liminv F_i$.

We regard a general frame 
$(W, (R_\lambda: \lambda \in \Lambda), P)$ as a first-order structure whose domain is the disjoint union of $W$ and $P$, with unary relations defining $W$ and $P$ and binary relations $R_\lambda \subseteq W \times W$ and 
$\epsilon \subseteq W \times P$ interpreted in the natural way. It is easy to write down a finite set $\Delta$
of first-order sentences expressing that a structure for this signature is a general frame.

As is well known, every modal formula $\phi$ has a standard translation to a formula $ST_x(\phi)$ of first-order
logic with a free variable $x$. We modify this here by regarding propositional variables as first-order variables. For a propositional variable p, we define $ST_x(p)$ to be $x\, \epsilon\, p$. We put $ST_x(\top) = \top$, $ST_x(\phi \land \psi)$ and similarly for negation, $ST_x(\b_\lambda \phi) = \forall y (R_\lambda(x,y) \to ST_y(\phi))$ and $ST_x(\d_\lambda\phi) = \exists y (R_\lambda(x,y) \land ST_y(\phi))$, 
where $y$ is a new variable. For a formula $\phi(p_1, \ldots, p_n)$, we write
$ST(\phi)$ for the universal closure $\forall x \in W \forall p_1 \dots p_n \in P \ ST_x(\phi)$. For a set $X$ of modal formulas we write $ST(X)$ for $\{ST(\phi) : \phi \in X\}$.  Clearly, a modal formula $\phi$ is valid in a general frame $\mathcal{G}$ iff $ST(\phi)$ is true in it in first-order semantics:
\begin{equation}\label{h10}
\mathcal{G} \models \phi \iffa \mathcal{G} \models ST(\phi).
\end{equation}

Hence, $\phi$ is valid in a Kripke frame $\F$ iff $ST(\phi)$ is true in $\F^+$ in first-order
semantics: 
\begin{equation}\label{h11}
\F \models \phi \iffa \F^+ \models ST(\phi).
\end{equation}

\begin{lemma}[Lemma 4.2 from \cite{GoldblattHodkinson2006}]\label{Canonical}
Let $\Fmc=(W, (R_\lambda : \lambda \in \Lambda), P)$ be a descriptive frame with countable $P$, and $\phi$ be 
a canonical formula. Then $\Fmc \models \phi$ implies $\Fmc_+ \models \phi$.
\end{lemma}

The following lemma generalises the argument of (and is inspired by) Theorem~4.4 from \cite{GoldblattHodkinson2006}, but seemingly does not follow from that paper.
 
\begin{lemma}\label{l16}
Let $\gamma_i$ be a sequence of modal formulas such that $\gamma_{i_1}$ implies $\gamma_{i_2}$ if $i_2 < i_1$.
Suppose that for all $l$ there exists $n$ such that for all $k$  there
exists an inverse system of finite Kripke frames $\{\F_i\}$ such that:
\begin{enumerate}[({L}1)]
\item for all $i$ $\F_i \models \gamma_k$,
\item $\liminv \F_i \models \gamma _l$,
\item $\liminv \F_i \not \models \gamma_n$.
\end{enumerate}
Then any axiomatisation of $L = \K + \{\gamma_n : n \in \omega\}$ has infinitely many non-canonical axioms.
\end{lemma}
\begin{proof}
Suppose on the contrary that $L$ is axiomatised by a single axiom $B$ together with a set $\Sigma$ of canonical formulas. Since $\Sigma \cup \{B\}$ and $\{\gamma_k : k < \omega\}$ axiomatise the same logic, the two 
first-order theories $\Delta \cup ST(\Sigma \cup \{B\})$ and $\Delta \cup \{ST(\gamma_k) : k < \omega\}$
have the same models. Then by the first-order compactness we conclude:
\begin{enumerate}[(a)]
\item there is $l < \omega$ such that $\Delta \cup ST(\gamma_l) \models ST(B)$,\\
 since $l$ is fixed, we fix $n$ from the condition of lemma, then
\item there is a finite subset $X \subseteq \Sigma$ such that $\Delta\cup ST(X \cup \{B\}) \models ST(\gamma_n)$
\item there is a finite $k$ such that $\Delta \cup ST(\gamma_k) \models ST(X)$, without loss of generality, we may take $k > l$.
\end{enumerate}

The condition of the lemma gives us an inverse system $\{\F_i\}$.
Let $\Fmc = (W, (R_\lambda : \lambda\in\Lambda), P) = \liminv (\F_i^+)$.
Clearly,  $\Fmc_+ = \liminv \F_i$.
By Proposition~\ref{fact}, since all $\F_i \models \gamma_k$, $\Fmc \models \gamma_k$. Plainly, $\Fmc \models \Delta$. 
Now, by (c) and (\ref{h10}) we obtain that $\Fmc \models X$. The formulas in $X$ are assumed canonical, and $P$ by
construction is countable,  therefore by
Lemma \ref{Canonical} $\Fmc_+ \models X$ as well. By (\ref{h11}), $(\Fmc_+)^+ \models ST(X)$.

As $\Fmc_+ \models \gamma_l$, (\ref{h11}) gives $(\Fmc_+)^+ \models ST(\gamma_l)$. Clearly, 
$(\Fmc_+)^+ \models \Delta$. So by (a), $(\Fmc_+)^+\models ST(B)$. 
Now we have 
$(\Fmc_+)^+\models \Delta \cup ST(X\cup \{B\})$, so by (b) and (\ref{h11}) we arrive at 
$\Fmc_+ \models \gamma_n$,  a contradition to (L3).
\end{proof}

\section{Axiomatisation}\label{section:axiomatization}

Fix a rooted  diagram $\D$. 
An axiomatisation of its normal modal logic $L^\D = \Log(\forall x_0 \E^\D(x_0))$ can be obtained 
using the algorithm from \cite{Hodkinson}.
It allows one to write modal axioms
for any normal modal logic generated by a first-order formula $\phi$ preserved under p-morphic images, disjoint unions and generated submodels, and this is our case.  This algorithm is followed quite liberally, since we act within a very restricted class of formulas, and this allows us to keep the presentation simpler and closer to the ultimate goal of this paper.  We also give an independent and `handmade' proof of the soundness and completeness and invite connoisseurs to compare it with the general  machinery from the quoted paper. They will definitely note that in terms of \cite{Hodkinson} the set $\Psi$ below is nothing else but a \emph{display} and $\gamma^\D_{\Psi}$ an \emph{approximant} for the hybrid formula $\eta^\D$. We also note that our axioms and completeness proof are similar to those in \cite{Venema98} for the ``reflexive successor'' logic.  

We proceed in two stages:  first, we construct a `hybrid equivalent' of $\phi$, second, we translate these hybrid formulas into modal axioms. 
To translate $\E^\D(x_0)$ into hybrid language, we need to construct a \emph{spanning tree}\footnote{Traditionally, spanning trees are defined for unoriented graphs and are unoriented. Here we use an oriented modification of this notion, but we still call it a spanning tree.} for $\D$.

\begin{defi} 

A tuple $\T = (W, (R_\lambda:\lambda\in\Lambda),r)$ is called a {\sl tree with a root $r$} if the
following holds
 
1) $r \in W$,

2) $R_\lambda^{-1} (r) = \emptyset$ for all $\lambda \in \Lambda$, where 
$R_\lambda^{-1} (x) = \{ z \mid (z,x) \in R_\lambda\}$,

3) for all  $x\in W \setminus \{r\}$ there is a unique directed path from $r$ to $x$. 

A tuple $\T= (W^\D, (R^\T_\lambda:\lambda\in\Lambda), x_0)$ is called a {spanning tree}
for a diagram\\ $\D= (W^\D, (R^\D_\lambda:\lambda\in\Lambda), x_0)$, if $\T$ is a tree, and
for all $\lambda\in\Lambda$, $x, y \in W^\D$ $x R^\T_\lambda y$  implies $x R^\D_\lambda y$.
\end{defi}

\begin{prop}[e.g., Lemma 5.5 from \cite{KZ}]
For any rooted diagram $\D$ there exists a spanning tree $\T$ for $\D$.
\end{prop}

Now, to every $x_i$ we assign a nominal $j_i$ and the hybrid formula
$$\chi^\D_i = j_i \land \bigwedge_{x_i R^\D_\lambda x_k} \d_\lambda j_k.$$ 
Then, working by induction on $\T$,  moving from leaves to the root,
to any point  $x_i$ of  $W^\D$  we assign a hybrid formula
$$\eta^\D_i = \chi^\D_i \land \bigwedge_{x_i R^\T_\lambda x_k} \d_\lambda \eta^\D_k.$$ 

Put $\eta^\D = \eta^\D_0$. Now note that $\E^\D(x_0)$ is equivalent to $\exists j_1 \ldots \exists j_n \eta^\D$ in that sense that for any pointed Kripke frame $\F$ and its point $w$ we have 
$\F, w \models \eta^\D$ iff $\F \models \E^\D(w)$.
By  $\eta^\D(\phi_0, \phi_1, \ldots, \phi_n)$ we denote the result of the substitution 
of modal formulas $\phi_l$ for nominals $j_l$ in the formula $\eta^\D$ for $0 \le l \le n$. 
For a set of formulas $\Psi$ and a map
$\kappa : \{0, 1, \ldots, n\} \to  \Psi$ let 
$\eta^\D(\kappa)  = \eta^\D(\kappa(0), \kappa(1), \ldots, \kappa(n))$. 
Then we set $$\gamma^\D_\Psi = \bigvee_{\kappa: \{0,1, \dots, n\} \to \Psi} \eta^\D( \kappa),$$
where the disjunction is taken over all possible maps $\kappa: \{0,1, \dots, n\} \to \Psi$.
Finally, let  $\Psi_h  = \{ \bar p ^ \varepsilon \mid \bar p = 
\{p_1, \ldots, p_h\}, \varepsilon \in \{0,1\}^h\}$, where $h < \omega $ 
and $\bar p^{\varepsilon_1\dots\varepsilon_h} = p_1^{\varepsilon_1}\land \dots \land p_h^{\varepsilon_h},$
$p^1 = p$, $p^0 = \neg p$. 
\begin{teo}\label{Theorem:Axiomatisation}
$L^\D$ is axiomatised by the set of formulas $\{\gamma_{\Psi_h}^\D \mid h \in \omega\}$.
\end{teo}

\begin{proof}{\it Soundness.} Given a Kripke frame $\F = (W^\F, (R^\F_\lambda : \lambda \in \Lambda))$, we show
that for any $h\in \omega$ and $w_0 \in W^\F$, $\F \models \E^\D(w_0)$ implies $F, w_0  \models \gamma_{\Psi_h}^\D$.
Suppose that $\F \models \E^\D(w_0)$, hence there exist points $w_1, \dots w_n \in W^\F$ such that
$\F \models \k^\D(w_0, w_1, \dots, w_n)$. Let $\theta$ be
a valuation on $\F$. For a fixed $i$, $0 \le i \le n$, let $\varepsilon(i)$ be the boolean vector, the $j$-th component of which
tells whether $w_i$ belongs to $\theta(p_j)$, and let $\kappa(i) = \bar p ^{\varepsilon(i)}$. It is easy to check that $\F, \theta, w_0 \models \eta^\D(\kappa)$.

{\it Completeness.}
Let $\F = (W^\F, (R^\F_\lambda:\lambda\in\Lambda))$ be the canonical frame for 
$\K +\{\gamma_{\Psi_h}^\D \mid h \in \omega\}$, i.e., $W^\F$ is the set of all  maximal  
$(\K +\{\gamma_{\Psi_h}^\D \mid h \in \omega\})$-consistent sets (mcs) of formulas, and 
for any $w,w' \in W^\F$, $w R^\F_\lambda w'$ 
iff for all formulas $\phi \in w'$ $\d_\lambda \phi \in w$. We show that $\F \models \forall x_0 \E^\D(x_0)$, and it follows that the logic 
$\K +\{\gamma_{\Psi_h}^\D \mid h \in \omega\}$ is Kripke complete with respect to the
elementary class defined by this formula, i.e.,$\forall x_0 \E^\D(x_0)$.

Take $w_0 \in W^\F$\hspace{-0.7mm}.\hspace{-0.3mm} Let us prove that there exist 
$w_1, \dots, w_n \in W^\F$ such that $\F \models \k^\D(w_0, \dots, w_n)$.
By $\V$ we denote the set of word tuples $\vv = (\v_0, \v_1, \dots, \v_n)$, where $\v_i \in \{0,1\}^*$.
We set $\vv \prec \vu$ iff for every $0 \le i \le n$ $\v_i$ is an initial segment of $\u_i$;
thus $(\V, \prec)$ is a transitive tree with branching $2^{n+1}$ in each node.
Suppose that $\psi_1, \psi_2, \psi_3, \dots$ is an enumeration of all modal formulas.
For a word $\v \in \{0,1\} ^*$, by $\v^j$ we denote the $j$-th symbol of $\v$ and we set
$$
\v^\# = \bigwedge_{j=1}^{|\v|} \psi_j^{\v^j},
$$
where for a modal formula $\psi$, $\psi^1 = \psi, \psi^0 = \neg \psi$.
By $\vv^\#$ we denote $\eta^\D(\v_0^\#, \v_1^\#, \dots, \v_n^\#)$. We say that $\vv$ is good if
$\F, \theta,w_0 \models \vv^\#$, where $\theta$ is the canonical valuation. 

Claim 1. If $\vv \prec \vu$ and $\vu$ is good, then so is $\vv$. Indeed, if $\vv \prec \vu$, then $\vu^\#$ implies $\vv^\#$.

Claim 2. For each $m$ there is a good $\vv$ such that for all $1 \le i \le n$ 
$|\v_i| = m$. To prove this claim, it is enough to notice that
$$\bigvee_{\vv : |\v_i|=m} \vv^\#$$
is a substitution instance of $\gamma_{\Psi_m}^\D$.

By K\"onig's lemma applied to $(\V, \prec)$  there exists a tuple of infinite strings 
$$\v^\inft=(\v^\inft_0,\v^\inft_1, \dots, \v^\inft_n),$$ such
that any tuple formed by the initial segments of its components is good. 
By setting  $w_i = \{\psi_j \mid (\v^\inft_i)^j = 1\}\cup\{\neg \psi_j \mid (\v^\inft_i)^j = 0\}$ for $1\le i \le n$,
the tuple $\vv^\inft$ gives rise to mcs'es $w_1, \dots, w_n$. It is clear that
$w_0 = \{\psi_j | (\v^\inft_0)^j = 1\}\cup\{\neg \psi_j | (\v^\inft_0)^j = 0\}$.

Now, we take $h$, $l$ and $\lambda$ such that $x_h R^\D_\lambda x_l$ and prove that $w_h R^\F_\lambda w_l$
in the canonical model. Suppose that for some $i$ $\psi_i \in w_l$ but $\d \psi_i \notin w_h$.
But $\b_\lambda \neg \psi_i = \psi_k$ for some $k$. Take $m = \max(i,k)$, and by
$\vv$ denote the word vector formed by the first $m$ bits of components of $\vv^\inft$.
Since $\vv$ is good, we have $w_0 \models \vv^\#$, which contradicts the consistency of $w_0$.
Indeed, since  $x_h R^\D_\lambda x_l$, $\eta^\D$ contains $j_h \land \d j_l$ as a subformula, maybe preceded by
diamonds and conjunctions, and thus $\vv^\#$ is built from $\{\psi_1, \dots, \psi_m\}$ using $\land$ and diamonds and has a subformula $\b_\lambda \neg \psi_i \land \d_\lambda \psi_i$.
\end{proof}

Let $d$ be the depth of the spanning tree $\T$ for $\D$ used in the construction of $\eta^\D$. 
We will also use formulas 
$\gamma^\D_m = \b^{\le d} (p_1 \lor \dots \lor p_m) \to \gamma_{\{p_1, \dots, p_m\}}^\D,$
which are more convenient to work with hereafter. 
Since $\gamma^\D_{\Psi_h}$ is equivalent to a substitution instance of 
$\gamma^\D_{2^h}$, we have
\begin{cor}
$L^\D$ is axiomatised by $\{\gamma^\D_m \mid m \in \omega\}$.
\end{cor}

Intuitively, $\gamma^\D_m$ says that
`if an $d$-neighborhood of a point $w_0$ of $\F$ is coloured in
$m$ colours, then we can paint $\tilde \T$ in $m$ colors such that 
the points of $\tilde \T$ with equal labels
have equal colours and
there exists a homomorphism from 
$\tilde \T$ to $\F$ preserving the colouring and sending $x_0$ to $w_0$, where $\tilde \T$ is a reduced syntactical tree of $\eta^\D$ 
defined as follows.

\begin{defi}
Let $\phi$ be a formula built using only $\land$, $\d_\lambda$ and nominals $j_k$ with $0 \le k \le n$.
A {\sl labelled tree with a root $r$} is a tuple $\tilde \T=(W, (R_\lambda:\lambda\in\Lambda), r, f)$, where  
$(W, (R_\lambda:\lambda\in\Lambda), r)$ is a tree with a root $r$ and
$f$ (a label function)
is a map from  $W$ to $\P(\{x_0, \dots, x_n\})$. 
A {\sl reduced syntactical tree} of a formula  $\phi$ 
is a labelled  tree $\tilde \T^\phi = (W^\phi,(R^\phi_\lambda:\lambda\in\Lambda), r^\phi, f^\phi)$ 
defined by induction on the length of $\phi$.

Case 1: $\phi = j_k$, where $k \in \{0, \dots, n\}$.
Then $W^\phi$ contains a single  
point $y$. The map $f^\phi$ takes $y$ to $\{x_k\}$ and the
relations $R_\lambda^\phi$ are empty.

Case 2: $\phi = \chi \land \psi$. Then put $W^\phi = 
(W^\chi \backslash \{r^\chi\}) \cup (W^\psi\backslash \{r^\psi\}) \cup\{ r^\phi\}$,
where $r^\phi$ is a new point.
The relations $R_\lambda$ on $W^\chi$ and $W^\psi$ remain the same, and
$r^\phi R_\lambda w$ iff $w\in W_\chi$ and $r^\chi R^\chi_\lambda w$
or $w\in W^\psi$ and $r^\psi R^\psi_\lambda w$. The map $f^\phi$ sends $r^\phi$ to 
$f^\chi(r^\chi) \cup f^\psi(r^\psi)$ and is equal to $f^\chi$ or  $f^\psi$ on
all other points.

Case 3: $\phi = \d_\lambda \psi$. Then $W^\phi$ = $W^\psi\cup \{ r^\phi\}$, where $r^\phi$ 
is a new point. The $R_\mu$ for $\mu \neq \lambda$ we leave untouched, and 
to $R_\lambda$ we add an arrow, joining $r^\phi$ with $r^\psi$. We put $f(r^\phi) = \emptyset$,
and do not change $f$ in all  other points.
\end{defi}

From the definition of $\tilde \T$ and $\eta^\D$ it follows that the label function of $\tilde \T$ 
maps the points of $\tilde \T$ to
singletons, and so it can be understood as a homomorphism from $\tilde \T$ to $\D$.
The labelled tree $\tilde \T$ may be also understood as a sort of unravelling of the initial diagram $\D$.

\begin{ex}\label{ex:1}
Let $\D$ be as in the Figure~\ref{fig:2.1} on the left. The spanning tree $\T$ 
is in the middle of the Figure~\ref{fig:2.1}, and so
$\eta = j_0 \land \d (j_2 \land \d j_1) \land \d (j_1 \land \d j_2)$, the reduced 
syntactical tree of which is  in the Figure \ref{fig:2.1} on the right. Thus
the logic $L^\D$ is axiomatized by the formulas 
$$\gamma^\D_m = \b (p_1 \lor \dots \lor p_m) \to \bigvee_{i,j = 1}^m (\d (p_i\land \d p_j) \land 
\d(p_j \land \d p_i)).$$

\begin{figure}[t]
\centering
\includegraphics{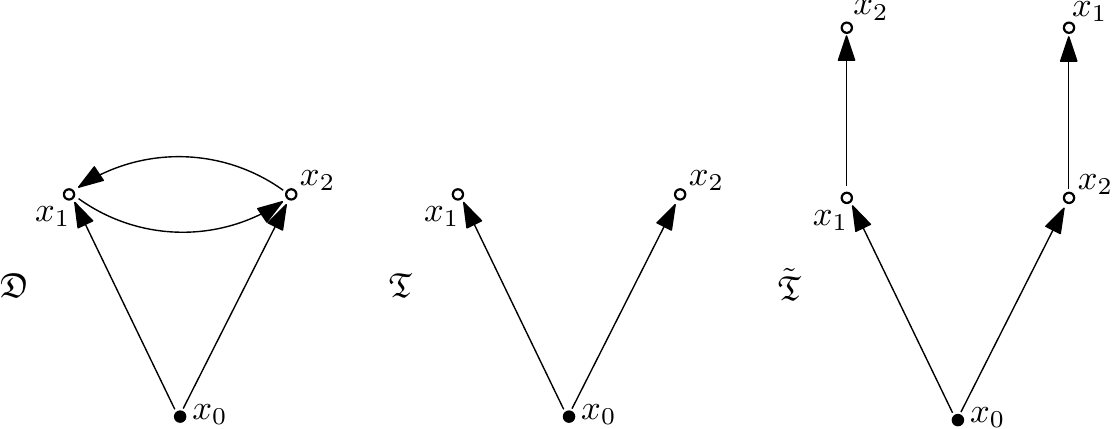}
\caption{A diagram $\D$ and its 1-unravelling $\tilde \T$.}
\label{fig:2.1}
\end{figure}
\end{ex}

\section{A property of globally minimal diagrams}

Suppose that $\D$ and $\D'$ are diagrams. A map $f: W^\D \to W^{\D'}$ is called a \emph{monotone map} from $\D$ to $\D'$, if for all $x, y \in W^\D$ and  $\lambda \in \Lambda$, 
$x R^\D_\lambda y$ implies $f(x) R^{\D'}_\lambda f(y)$. If also  $f(x^\D_0) = x^{\D'}_0$, then $f$ is called
a \emph{homomorphism}. 

\begin{lemma}\label{rank1}
For any globally minimal rooted diagram  $\D$ with an inner cycle, there exist two pointed finite Kripke frames
$\F^\D_+ = (W^\D_\pm,  ((R^\D_+)_\lambda : \lambda \in \Lambda), w_0)$ and
$\F^\D_- = (W^\D_\pm,  ((R^\D_-)_\lambda : \lambda \in \Lambda), w_0)$, points $x_d, x_{d'} \in W^\D$, 
an index $\lambda_d \in \Lambda$, and an injective homomorphism 
$g: \D \to \F^\D_+$ sending $x_0$ to $w_0$ such that:
\begin{enumerate}[(\it C-i)]
\item $\F^\D_- = \F^\D_+ - (g(x_d), g(x_d'), \lambda_d);$
\item $\F^\D_- \not\models \E^\D(w_0)$;
\item $\F^\D_+ \models \E^\D(w_0)$;
\item the points $g(x_d)$ and $g(x_{d'})$ can be connected in  $\F^\D_-$ by an undirected 
path not passing through $w_0$, all points of which belong to the image of $g$;
\item for any homomorphism $h$ from $\D$ to $\F^\D_+$ we have
\begin{equation}\label{5.1}
\mbox{the image of $h$ is } \{g(x_0), \ldots, g(x_n)\},
\end{equation} and for all $0 \le i, j \le n$ and $\lambda \in \Lambda$ 
\begin{equation}\label{5.2}
h(x_i) (R^\D_+)_\lambda h(x_j) \mbox{ implies } x_i R^\D_\lambda x_j;
\end{equation}
\item if $w \neq w_0$, then  $\F^\D_- \models \E^\D(w)$ for all  $w \in W^\D_\pm$.
\end{enumerate}
\end{lemma}

({\it C-iv}) and ({\it C-v}) are technical conditions needed to prove that, for example, 
the pseudoproducts constructed in Section~\ref{GraphPseudoproducts} refute $\E^\D(x_0)$ in their roots. In fact,
({\it C-v}) says that any homomorphism $h$ from $\D$ to $\F^\D_+$ is an isomorphism between $\D$ and 
the restriction of $\F^\D_+$ to the image of $h$, and
it is easy to see that ({\it C-v}) together with ({\it C-i})  always imply ({\it C-ii}). Indeed, suppose that $\F^\D_- \models \E^\D(w_0)$. Hence, there is
a homomorphism $h$: $\D \to \F^\D_-$. But $h$ is also a homomorphism from $\D$ to $\F^\D_+$. Thus, by ({\it C-v}),
$h$ is an isomorphism between $\D$ the restriction of 
$\F^\D_+$ to $h(W^\D)$. Therefore there must be points $x_i,x_j \in \D$,
such that $h(x_i) = g(x_d)$, $h(x_j) = g(x_{d'})$, and $(x_d, x_{d'}) \in R^\D_{\lambda_d}$. 
Now we have a contradiction to the facts that $h$ is a homomorphism to $\F^\D_-$ and  
$(g(x_d), g(x_{d'})) \notin (R^\D_-)_{\lambda_d}$.

\begin{figure}[t]
\centering
\includegraphics{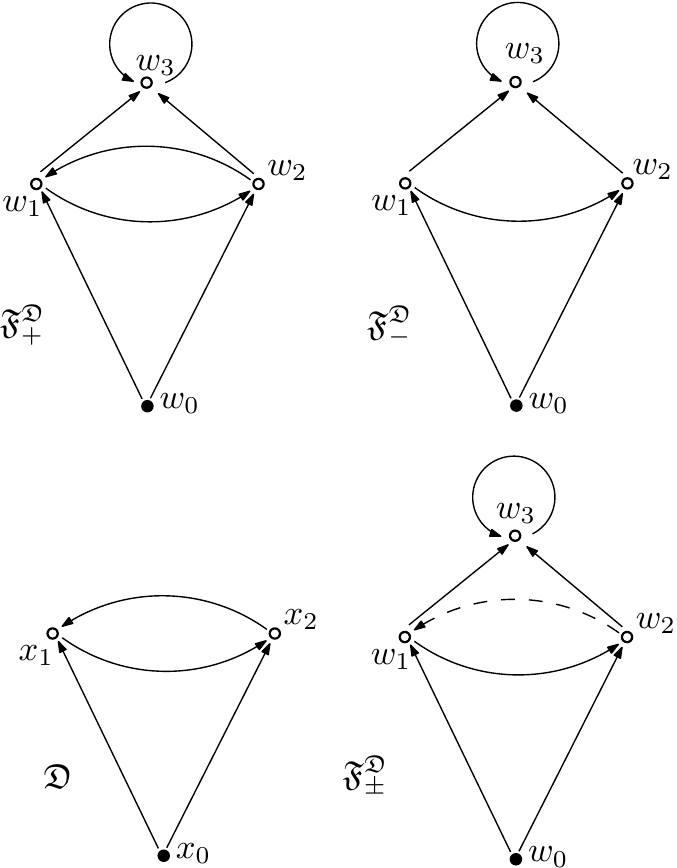}
\caption{}
\label{fig:4}
\end{figure}

\begin{ex}\label{ex:2}
Consider the diagram from Figure~\ref{fig:4}. It is easy to see that the frames $\F^\D_+$ and $\F^\D_-$ from this figure satisfy ({\it C-i}) -- ({\it C-vi}), where $d = 2,$ $d' = 1,$ and $g(x_i) = w_i$ for $i = 0, 1, 2$. 
In particular, there exist two homomorphisms from $\D$ to $\F^\D_+$: 
the first coincides with $g$, the second swaps around $x_1$ and $x_2$, and both of them satisfy ({\it C-v}). Since $\F^\D_+$ and $\F^\D_-$ have a common base set $W^\D_\pm$, 
one can think of $\F^\D_+$ and $\F^\D_-$ as a tuple $\F^\D_\pm$, consisting of the frame 
$\F^\D_+$ and a selected edge given by $d, d'$ and $\lambda_d$.
\end{ex}

In general, to satisfy conditions ({\it C-i}) -- ({\it C-v}), we can always build a spanning tree for $\D$, and take $\F^+ = \D$ and $\F^- = \F^+ - (x_d, x_{d'}, \lambda_d)$, where $(x_d, x_{d'}, \lambda_d)$ is
one of the edges of the inner cycle not belonging to the spanning tree. 
The main problem is condition ({\it C-vi}). In this example we got it at the price of a reflexive point on
top of $\D$. But in some cases 
this may break the conditions ({\it C-v}) and ({\it C-ii}), as the
next example shows,  so a more subtle construction is required.

\begin{ex}
Consider the diagram $\D$ on the left hand side of Figure~\ref{fig:3}. It is minimal. But if we choose the edge to delete 
(it could be done in a unique way without affecting connectivity; this edge is dashed in the figure in the middle), add a reflexive point and connect all points except the root to this reflexive point, then the obtained 
diagram (even after removing the selected arrow)
will satisfy $\E_\D(x_0)$ (see the frame in the middle), and that is bad. But we can amend this situation by a more elaborate construction as on the right hand side. Roughly, at first we iteratedly add new points by considering  $\E_\D(x_0)$ as a tuple-generating rule, and only after that we close the construction with a reflexive point.
\end{ex}

\begin{figure}[t]
\centering
\scalebox{1.2}{\includegraphics{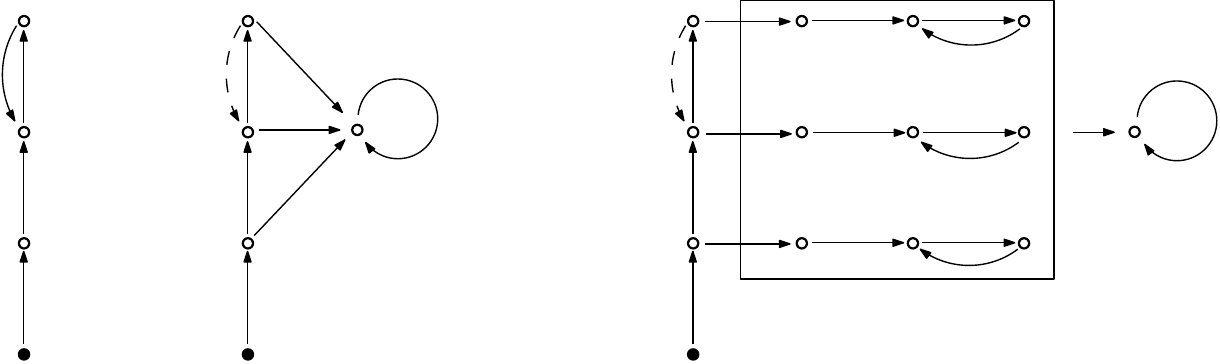}}
\caption{}
\label{fig:3}
\end{figure}
\suppressfloats

\begin{proof}[The proof of Lemma~\ref{rank1}]
Given a rooted diagram $\D = (W^\D, (R^\D_\lambda : \lambda \in\Lambda), x_0)$, we define a sequence
of tuples $(\F^\D_i, A^\D_i)$ where $\F^\D_i = (W^\D_i, R^\D_{i,\lambda}, x^\D_{i,0})$ is a Kripke frame and
$A^\D_i \subseteq W^\D_i$. 
Intuitively, $A_i$ denotes the set of  those points $w$ in $\F_i$ which
may falsify ({\it C-vi}) and so have to be ``repaired'', which results in $\F_{i+1}$.
Formally, we set  $\F^\D_1 = \D$ and $A^\D_1 = W^\D \setminus \{x_0\}$, $x^\D_{1,0} = x_0$. 
Now given $(\F^\D_i, A^\D_i)$, we define
$(\F^\D_{i+1}, A^\D_{i+1})$ as follows. Let
$$W^\D_{i+1} = ((W^\D_i \setminus A_i) \times \{x_0\}) \cup (A_i \times W^\D),$$
$$R^\D_{i+1,\lambda} = \{((a,x_0), (b,x_0)) \mid (a,b) \in R^\D_i \} \cup \{((a,b), (a,c)) \mid a \in A_i,\, b,c \in W^\D \mbox{ and } 
b R^\D_\lambda c\},$$
$$A^\D_{i+1} = A^\D_i \times (W^\D \setminus \{x_0\}) \mbox{ and } x^\D_{i+1,0} = x^\D_{i,0} \times \{x_0\}.$$

Let $r$ be the maximal distance from $x_0$ to any point $x_i$ of $\D$.
To obtain $\F^\D_+$, we take $\F^\D_r$, add a new all-$\lambda$-reflexive point $\circ$, and join   
all points $y$ of $A_r$ to $\circ$  by all $R_\lambda$. 
Now we set $g(x_i) = (x_i, x_0, \ldots, x_0)$ and  
$\F^\D_- = \F^\D_+ - (g(x_d), g(x_{d'}), \lambda_d)$, where $(x_d, x_{d'}, \lambda_d)$ is an arbitrary chosen edge  of the inner cycle that does not belong to some spanning tree for $\D$.
Denote $\Delta = \{g(x_0), \ldots, g(x_n)\} = 
W^\D \times \{x_0\}\times \ldots \times \{x_0\}$. Thus,  $g$ is a bijection between
$W^\D$ and $\Delta$.

We claim that the conditions ({\it C-i})--({\it C-vi}) are satisfied. 
It is clear that the construction guarantees the validity of conditions ({\it C-i}), ({\it C-iii}) 
and ({\it C-vi}). 
Condition ({\it C-iv}) is true since 
$(x_d, x_{d'}, \lambda_d)$ belongs to the inner cycle.

Let us prove ({\it C-v}). Suppose that there exists a homomorphism $h$ from $\D$ to  $\F^\D_+$.
It is clear that $h(x_i) \neq \circ $ for all $i$, 
because $r$ was chosen large enough and so 
the new reflexive point $\circ$ is too far from $w_0$.
Let $\Gamma = \Delta \cap \{ h(x_0), h(x_1) , \ldots, h(x_n)\}$. 
Suppose that $\Delta \setminus \Gamma$ is not empty.

For $x_i \in W^\D$ let $\rank(x_i)$ denote the distance from $x_0$ to $x_i$ in $D$. 
Let $\Del(x_i)$  be the set of all $x_j \in W^\D$, $x_j \neq x_i$ such that 
all paths from $x_0$ to $x_j$ pass through $x_i$. Informally speaking, if we delete $x_i$, then
$\Del(x_i)$ is the set of all points that cease to be visible from $x_0$.

Now let $x_\alpha$ be a point of $W^\D$ of maximal rank such that 
$g(x_\alpha) \in \Delta  \setminus\Gamma$.  Let $\D'$ be obtained from $\D$ by deleting 
$x_\alpha$ together with all adjacent edges.
The following claims show that $\D'$ 
is rooted and that the image of $h$ belongs to 
$\F^{\D'}_r$ considered as a part of $\F^\D_r$.

\emph{Claim 1.} $\Del(x_\alpha) = \emptyset$. Otherwise,
take a point $x_{\alpha'}$ of $\Del(x_\alpha)$.
From the definition of $\Del(x_\alpha)$ and $\F^\D_r$ it follows that all paths in $\F^\D_+$ leading from $g(x_0)$ to 
$g(x_{\alpha'})$ pass through $g(x_\alpha)$. Let us prove that $g(x_{\alpha'}) \notin \Gamma$. 
Suppose that  $g(x_{\alpha'}) \in \Gamma$, that is $g(x_{\alpha'}) = h(x_{\alpha''})$ for some $\alpha''$. 
Since $\D$ is rooted,  there exists a path in $\D$ from $x_0$ to $x_{\alpha''}$. The image of this path under the map $h$ is a path connecting $g(x_0)$ to $g(x_{\alpha'})$ in $\F^\D_+$. But this path must pass through $g(x_\alpha)$. That contradicts $g(x_\alpha) \notin \Gamma$.
Hence, $g(x_{\alpha'}) \notin \Gamma$. But $\rank(x_{\alpha'}) > \rank (x_\alpha)$. This contradicts 
the maximality of the rank of $x_\alpha$ in $\Delta \setminus \Gamma$. 

\emph{Claim 2.} If $h(x_\alpha) = (x_{i_1}, \ldots, x_{i_n})$, then 
 $x_{i_k} \neq x_\alpha$ for all $1 \le k \le n$. First, consider the case when $h(x_\alpha) \in \Delta$.
Then there is nothing to prove, because in this case $i_k = 0$ for $2 \le k \le n$, and $i_1 \neq \alpha$ 
since $g(x_\alpha) \notin \Gamma$. So, suppose that 
$h(x_\alpha) \notin \Delta$. This means that 
\begin{equation}\label{ge2}
|\{k \mid x_{i_k} \neq x_0 \} | \ge 2.
\end{equation}
Let us prove that  for all 
$1 \le k \le n$ we have $\rank(x_{i_k}) < \rank(x_\alpha)$.  
To this end, note that the distance in $\F^\D_+$ satisfies 
$$\dist_{\F^\D_+} ((x_0, \ldots, x_0),( x_{i_1}, \ldots, x_{i_n})) = 
\dist_{\D}(x_0, x_{i_1}) + \ldots + \dist_{\D}(x_0, x_{i_n}).$$
Now, suppose that for some $k$ we have $\rank(x_{i_k}) \ge \rank(x_\alpha)$. Then taking into account (\ref{ge2})
we obtain that  
$$\dist_{\F^\D_+}(w_0, h(x_\alpha)) = 
\dist_{\F^\D_+} ((x_0, \ldots, x_0),( x_{i_1}, \ldots, x_{i_n})) > 
\rank(x_\alpha),$$
a contradiction. Hence, $\rank(x_{i_k}) < \rank(x_\alpha)$ for all $1 \le k \le n$, and, in particular, 
$x_{i_k} \neq x_\alpha$.

\emph{Claim 3.} For any $\beta \in \{1, \dots, n\}$, if $h(x_\beta) = ( x_{i_1}, \ldots, x_{i_n})$, then 
$x_{i_k} \neq x_\alpha$ for all $1 \le k \le n$. Indeed, if $\rank(x_\beta) \le \rank(x_\alpha)$, then a similar 
argument works. Now, suppose that $\rank(x_\beta) > \rank(x_\alpha)$. Let $B \subseteq W^\D$ be the set of
all points of rank greater than $\rank(x_\alpha)$. But as $x_\alpha$ is supposed to be an element of 
$\Delta\setminus\Gamma$ of maximal rank, from the definition of $B$ it follows 
that $h(B) \subseteq \Delta$, and, in particular,
for all $k>2$ $x_{i_k} = x_0 \neq x_\alpha$. 

Now, let the diagram $\D'$ be obtained from $\D$ by deleting $x_\alpha$ together with all adjacent edges. 
Then  $\vdash_{FOL} \forall x_0 \E^{\D'}(x_0) \to \forall x_0 \E^{\F^{\D'}_r} (x_0)$ and
$\vdash_{FOL} \forall x_0 \E^{\F^{\D'}_r} (x_0) \to \forall x_0 \E^{\D} (x_0)$. Thus 
$\vdash_{FOL} \forall x_0 \E^{\D'}(x_0) \to \forall x_0 \E^{\D} (x_0)$, and this
contradicts the  global minimality of $\D$. 

We have just proved (\ref{5.1}) of ({\it C-v}).  To prove (\ref{5.2}), take a homomorphism $h$ from
$\D$ to $\F_\D^+$. Via identification the $\D$ with a copy of itself sitting inside $\F_\D^+$ given 
by the image of $g$, we see that the map $h$ acts on the set 
$\Arr(\D) = \{(x_i, x_j, \lambda) \mid x_i, x_j \in W^\D, \lambda \in \Lambda, (x_i,x_j) \in R^\D_\lambda\}$, 
sending $(x_i, x_j, \lambda)$ to
$(h(x_i), h(x_j), \lambda)$ which is also in  $\Arr(\D)$.  From (\ref{5.1}) it follows that 
$h$ is injective on $\Arr(\D)$. Therefore, since $\Arr(\D)$ is finite,  
$h$ is surjective on $\Arr(\D)$, and so satisfies  (\ref{5.2}) of ({\it C-v}).

Another proof of  (\ref{5.2}) of ({\it C-v}): take a homomorphism $h$ from
$\D$ to $\F_\D^+$ and assume that $h(x_i) (R^\D_+)_\lambda h(x_j)$ holds 
while $x_i R^\D_\lambda x_j$ does not. It follows that the edge 
$g^{-1}(h(x_i)) R^\D_\lambda  g^{-1}(h(x_i))$ (which is well defined because of 
(\ref{5.1})) is superfluous in $\D$.

Condition ({\it C-ii}) is a consequence of  ({\it C-v}) and ({\it C-i}).
\end{proof}

\section{Pseudoproducts with graphs} \label{GraphPseudoproducts}

By a \emph{graph} we understand a tuple $G = (V, E)$, where $E$ is a symmetric binary relation on $V$.
To emphasise symmetricity of $E$, instead of $(v_1, v_2) \in E$ we sometimes write $\{v_1, v_2\} \in E$.
 For a ordinal $\alpha \le \omega$ an $\alpha$-colouring of a graph $G$ is
a map $\tau: V \to \alpha$, such that every two adjacent vertices are mapped to different elements of 
$\alpha$. The elements of $\alpha$ in this context are called \emph{colours}.
In general, below we do not suppose $E$ to be irreflexive, thus $G$ may contain edges of the form
$(v,v)$, which are called \emph{loops}. 
However, it is clear that any graph that contains loops does not have colourings at all, and so  
in Lemma~\ref{colouring}, (C2) we implicitly assume that $G$ does not have loops.

We fix a minimal diagram $\D$ with an inner cycle. 
Let $\F^\D_+= (W^\D_\pm, ((R^\D_+)_\lambda:\lambda \in \Lambda), w_0)$ and $\F^\D_- = (W^\D_\pm, ((R^\D_-)_\lambda:\lambda \in \Lambda), w_0)$
together with $d$, $d'$, $\lambda_d$ and $g$ satisfy conditions ({\it C-i}) -- ({\it C-vi})  of Lemma \ref{rank1}. 
Let $G = (V, E)$ be an arbitrary graph.
By $\F^\D_{\pm}\times G$ denote\footnote{We chose this notation for pseudoproducts because they somehow combine the features
of $\F^\D_+$ and $\F^\D_-$. You may think of $\F^\D_\pm$ as a shorthand for a tuple $(\F^\D_+, \F^\D_-)$ with intuition from Example~\ref{ex:2}.} the Kripke frame $(W^{\D,G}, (R^{\D,G}_\lambda : \lambda \in \Lambda))$ where
$W^{\D,G} = \{w_0\} \cup (W^\D_\pm \setminus \{w_0\}) \times V$ and

\begin{equation*}
R_\lambda^{\D,G} = \{(w_0,w_0)\mid \F^\D_- \models w_0 R_\lambda w_0; \} \cup 
\end{equation*}
\begin{equation*}
\{(w_0,(y,v))\mid \F^\D_- \models w_0 R_\lambda y; y\in W^\D_\pm\setminus\{w_0\}; v\in V\} \cup 
\end{equation*}
\begin{equation*}
     \{((y,v),w_0)\mid \F^\D_- \models y R_\lambda w_0; y \in W^\D_\pm\setminus\{w_0\}; v\in V \} \cup 
\end{equation*}
\begin{equation*}
     \{((x,v),(y,v))\mid \F^\D_- \models x R_\lambda y; x,y \in  W^\D_\pm \setminus\{w_0\}; v\in V; \} \cup
\end{equation*}
\begin{equation*}
   \{((g(x_d),v_1),(g(x_{d'}),v_2))\mid v_1\in V; v_2 \in V;  \{v_1,v_2\} \in E, \lambda = \lambda_d\}.
\end{equation*} 

An example of this construction for $\D$ and $\F^\D_\pm$ from Example~\ref{ex:2} is given in Figure~\ref{fig:7}.
\begin{figure}[t]
\centering
\includegraphics{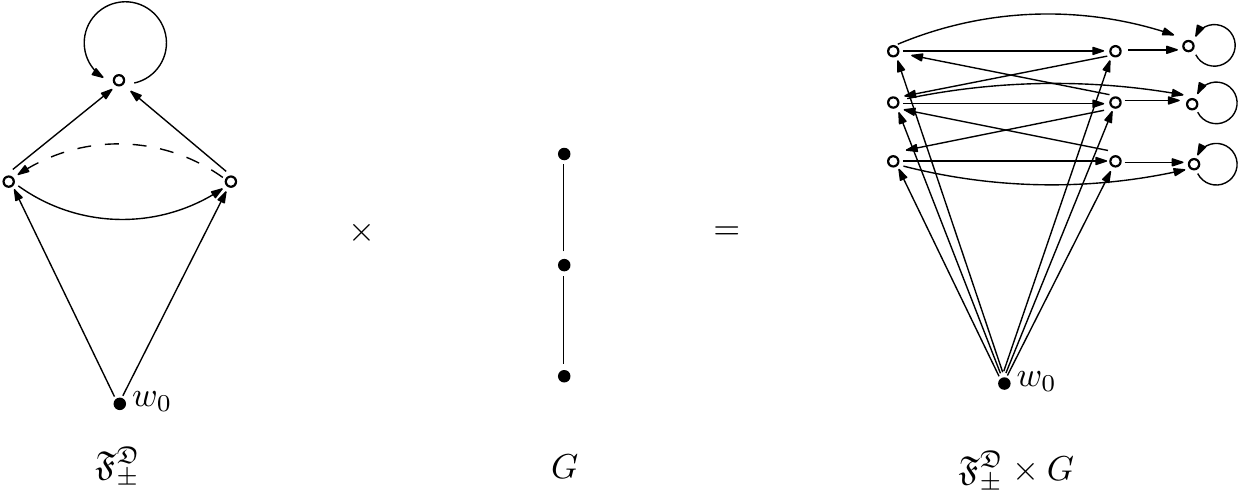}
\caption{A pseudoproduct.}
\label{fig:7}
\end{figure}
This construction has a simpler description in terms of projection functions. If $\pr$ denotes the projection
from $\F^\D_\pm \times G$ to $\F^\D_+$, given by formulas $\pr((x,v)) = x$, $\pr(w_0) = w_0$, and $h$ denotes
the projection from $\F^\D_\pm \times G$ to $V'$, where 
$V'=V \cup \{\bot\}$, given by formulas $h((x,v)) = v$, 
$h(w_0) = \bot$, then the $R^{\D,G}_\lambda$ satisfy the following condition for all $\eta, \chi \in W^{\D, G}$: 
$\F^\D_\pm \times G \models \eta R_\lambda \chi$ iff
$$
\begin{array}{c}
\F^\D_-\models \pr(\eta)R_\lambda \pr(\chi) \mbox{ and } 
(h(\eta) = h(\chi) \mbox{ or } \bot \in\{h(\eta), h(\chi)\} )\\
\mbox{ or  }\\
\F^\D_-\not\models \pr(\eta)R_\lambda \pr(\chi), \F^\D_+\models \pr(\eta)R_\lambda pr(\chi) \mbox{ and } G \models h(\eta) E h(\chi).
\end{array}
$$

Recall that $L^\D$
is axiomatized by formulas $\gamma^\D_n$ of Section~\ref{section:axiomatization}
saying ``if an d-neighborhood of a point $w_0$ of $\F$ is painted in
$m$ colours, then we can paint the tree $\tilde \T = (W^\Tt, (R^\Tt_\lambda : \lambda \in \Lambda), x_0)$ (defined
in Section~\ref{section:axiomatization}) in $m$ colors such that the points of $\tilde \T$ with equal labels
have equal colours and there exists a homomorphism from 
$\tilde \T$ to $\F$ sending $x_0$ to $w_0$ and preserving the colouring''. 
We understand $f^\Tt$ as a homomorphism
from $\Tt$ to $\D$. 

The next lemma shows the connection between the chromatic number of $G$ and the least $k$ for which 
$\gamma^\D_k$ can be refuted in $\F^\D_\pm \times G$. 
Simply put, it says that if one of these numbers is big, 
then the other is big as well.

\begin{lemma}\label{colouring}
Suppose that $|W^\D_\pm| = b$ and that $G=(V,E)$. Then
\begin{enumerate}[({C}1)]
\item If $G$ cannot be painted in $2^{bk}$ colours, then for all $k$-generated valuations $\theta$ and for all $m$ we have $\F^\D_\pm \times G, \theta \models \gamma^\D_m$. In particular, $\F^\D_\pm \times G \models \gamma^\D_k$.
\item If $G$ can be painted in $N$ colours, then $\F^\D_\pm \times G \not \models \gamma^\D_{N(b-1)+1}$.
\end{enumerate}
\end{lemma}
\begin{proof}
(C1) Condition ({\it C-vi}) of Lemma~\ref{rank1} and  
the soundness part of Theorem~\ref{Theorem:Axiomatisation}
guarantee that $\F^\D_\pm \times G, x \models \gamma^\D_m$ for all $x \in W^{\D,G}$ different from $w_0$. 
We show that $\F^\D_\pm \times G, \theta, w_0 \models \gamma^\D_m$ as well, if $\theta$ is $k$-generated.
Since a formula $\gamma^\D_m$ is invariant under transpositions of variables that swap $p_i$ and $p_j$, without
any loss of generality we may assume that $\theta(p_i) = \emptyset$ for $i > k$.
Define the map $\tau : V \to (\P(\{1, \dots, k\}))^{(b-1)}$ by putting $\tau(v)$  for $v\in V$ to be the map
from $W^\D_\pm \setminus \{w_0\}$ to $\P(\{1, \dots, k\})$ defined by
$$i \in \tau(v) (y) \mbox { iff }  (y,v) \in \theta(p_i) \quad \mbox { for } 1\le i \le k. $$
Since $G$ cannot be painted in $2^{bk}$ colours, there exist $v_1, v_2 \in V$ such that $\{v_1, v_2\} \in E$ and
$\tau(v_1) = \tau(v_2)$. Consider the Kripke frame
$(\F^\D_\pm \times G)^\dagger = ((W^{\D,G})^\dagger, (R^{\D,G}_\lambda)^\dagger, w_0)$, 
where $(W^{\D,G})^\dagger = W^{\D,G} \cup W^\D_\pm$ 
(recall that $W^\D_\pm \cap W^{\D,G} = \{w_0\}$) and 
$(R^{\D,G}_\lambda)^\dagger = R^{\D,G}_\lambda \cup (R^\D_+)_\lambda$.
Then we set $$\theta^\dagger(p_i) = \theta(p_i) \cup 
\{y \in W^\D_\pm \mid (y,v_1) \in \theta(p_i)\}$$
for all $1 \le i \le k$.
We claim that $((\F^\D_\pm \times G)^\dagger, \theta^\dagger), w_0$ and  
$(\F^\D_\pm \times G, \theta), w_0$ are bisimilar. Indeed, the relation 
$Z = \{(z,z) \mid z \in W^{\D,G}\} \cup \{(y,(y,v_1)) \mid  y \in W^\D_\pm \setminus \{w_0\}\} \cup
\{(y,(y,v_2)) \mid  y \in W^\D_\pm \setminus \{w_0\}\}$ constitutes a bisimulation.
Since 
$(\F^\D_\pm \times G)^\dagger, \theta^\dagger,  w_0 \models \gamma^\D_m$ 
(Theorem~\ref{Theorem:Axiomatisation}, Soundness), 
we conclude that also $\F^\D_\pm \times G, \theta, w_0 \models \gamma^\D_m$.

(C2) Let $\tau: V \to \{1, \dots, N\}$ be a colouring of $G$. Suppose that the variables of 
$\gamma^\D_{N(b-1)+1}$ are 
indexed as $p_0$ and $p_i^c$ where  $1 \le i \le b - 1,$ $1 \le c \le N$. Consider the 
following valuation $\theta$ on 
$\F^\D_{\pm} \times G$:
$$
\theta(p) = \begin{cases}
\{x_0\},   & \mbox{ if } p = p_0,\\
\{(x_i, v) \mid \tau(v) = c\},& \mbox{ if } p = p_i^c.
\end{cases}
$$
The definition of $\theta$ gives rise to the  map $\theta_* : \F^\D_\pm \times G \to \{0,1,\dots, N(b-1)\}$ defined
by equations $\theta_*(w_0) = 0$; $\theta_*((y,v)) = $ the number of $p_i^{\tau(v)}$ among $\{1, \dots, N(b-1)\}$.

Recall that there is a natural projection $\pr: \F^\D_{\pm} \times G \to \F^\D_+$, defined by 
$$\pr(w_0) = w_0;$$
$$\pr(x_i, v) = x_i \text{ for all } v \in V.$$

Besides $\pr$, there is a projection $f^\Tt : \Tt \to \D$.\hspace{-1mm}
We say that a map $\mathfrak{b}\hspace{-0.5mm} :\hspace{-0.5mm} W^\Tt \to \{0,1, \ldots, N(b-1)\}$ \emph{respects $f^\Tt$} if for all 
$x, y \in W^\Tt$, $f^\Tt(x) = f^\Tt(y)$ implies $\mathfrak{b}(x) = \mathfrak{b}(y)$.

To prove that $\F^\D_\pm \times G, \theta \not \models \gamma^\D_{N(b-1)+1}$ it is sufficient to prove that there is no homomorphism 
$\a : \Tt \to \F^\D_\pm \times G$, such that $\theta_*(\a(t)) : \Tt \to \{0, 1, \dots, N(b-1)\}$ respects $f^\Tt$.

\begin{figure}[t]
\centering
\begin{tikzpicture}
\node(D) at (0,0) [label = below : ${z_1,z_2,\dots,z_s}$]{$\D$};
\node(T) at (0,3)  {$\Tt$};
\node(FG) at (3,3) {$\F^\D_\pm\times G$};
\node(F) at (3,0) [label = below : ${y_1,y_2,\dots,y_s}$] {$\F^\D_+$};
\node(C) at (3,6) {$\{0,1, \dots, N(b-1)\}$};
\draw[->] (T)  to node [above]{$\mathfrak{a}$} (FG);
\draw[->] (D)  to node [above]{$h$} (F);
\draw[->] (T)  to node [left]{$f^T$} (D);
\draw[->] (FG)  to node [right]{$pr$} (F);
\draw[->] (FG)  to node [right]{$\theta_*$} (C);
\draw[->] (T)  to node [above left]{$\mathfrak{b}$} (C);
\end{tikzpicture}
\caption{}
\label{fig:2}
\end{figure}

For the sake of contradiction, assume that such $\a$ exists. Consider the following map $h : W^\D \to W^\D_\pm$. 
To define $h(x)$ for  $x \in W^\D$, we take any $t \in (f^\Tt)^{-1}(x)$ and set $h(x) = \pr(\a(t))$ 
(see Figure~\ref{fig:2}). It is clear that $h$
is well defined, i.e., it does not depend on the choice of $t$, since if $t_1, t_2 \in (f^\Tt)^{-1}(x)$, then
$f^\Tt(t_1) = f^\Tt(t_2)$, and this means that $\pr(\a(t_1)) = \pr(\a(t_2))$ due to the definition of $\theta$.
Clearly, $h$ is a homomorphism from $\D$  to $\F^\D_+$, and it makes the diagram in Figure~\ref{fig:2} commutative.  Now we apply ({\it C-v}) of Lemma \ref{rank1} and conclude that  the image of  $h$ is $\{g(x_0), g(x_1), \ldots, g(x_n)\}$.

Then we apply  ({\it C-iv}). Let $y_1 R_{\lambda_1} y_2 \ldots   R_{\lambda_{n-1}} y_s$, 
where $y_i \in W^\D_\pm$ for $1 \le i \le s$ and $\lambda_i \in \Lambda^\pm$ for $1\le i < s$, 
be the path  connecting 
$x_{d}$ with $x_{d'}$ in $\F^\D_-$ (in particular, $y_1 = x_d$ and 
$y_s = x_{d'}$). Let $z_1, \dots, z_s \in W^\D$ be the points such that 
$h(z_i) = y_i$ for $1\le i \le s$. By ({\it C-v}), (\ref{5.2}), $\D \models z_i R_{\lambda_i} z_{i+1}$ for $1\le i < s$. The map $f^\Tt$ satisfies the following condition: if $\D \models w_1 R_\lambda w_2$ for some 
$w_1, w_2 \in W^\D$, then there exist points $w_1', w_2' \in W^\Tt$ such that 
$\Tt \models w_1' R_\lambda w_2'$ and  $w'_i \in (f^\Tt)^{-1}(w_i)$ for $i \in \{1,2\}$. 
We apply this statement $s-1$ times for $z_i$, $R_{\lambda_i}$ and $z_{i+1}$, and conclude that
there exist
points $t_i \in (f^\Tt)^{-1}(z_i)$ and $t'_j \in (f^\Tt)^{-1}(z_j)$ 
for $1 \le i < s $ and $ 1 < j \le s$ such that $\Tt \models t_i R_{\lambda_i} t'_{i+1}$ for 
$1 \le i < s$.  Note that by the definition of $t_i$ and $t_i'$ we have $\pr(\a(t_i)) = \pr(\a(t'_i)) = y_i$.

Thus, let $v_i \in V$ and $v'_j \in V$ 
for $1 \le i < s $ and $ 1 < j \le s$ be such that
$\a(t_i) = (y_i, v_i)$ and $\a(t'_j) = (y_i, v'_j)$. Let us show that
\begin{enumerate}[(a)]
\item if $ 1\le i < s$ then  $\tau(v_i) = \tau(v'_{i+1})$ and
\item if $ 1< i < s$ then  $\tau(v_i) = \tau(v'_{i})$.
\end{enumerate}

(a): Since $\Tt \models t_i R_{\lambda_i} t_{i+1}$, due to the definition of 
$R^{\D,G}$ and the facts that $(f^\Tt(t_i), f^\Tt(t_{i+1}))\in (R^\D_-)_{\lambda_i}$ and 
$(g(x_d), g(x_{d'})) \notin (R^\D_-)_{\lambda_d}$, 
it follows that $v_i = v'_{i+1}$, and so $\tau(v_i) = \tau(v'_{i+1})$.

(b): From $t_i, t'_i \in (f^\Tt)^{-1}(z_i)$ and the fact that $\theta(\a(t))$ respects $f^\Tt$, 
it follows that 
$\theta_*(\a(t_i)) = \theta_*(\a(t'_i))$. Therefore 
$\theta_*((y_i, v_i)) = \theta_*((y_i, v'_i))$, and so  $\tau(v_i) = \tau(v'_i)$.

Together, (a) and (b) give us that $\tau(v_1) = \tau(v_s')$. 
Since $F^\D_+ \models g(x_d) R_{\lambda_d} g(x_d')$, ({\it C-v})  implies that
$\D \models  z_1 R_{\lambda_d} z_s$, and so there exist $t'_1 \in (f^\Tt)^{-1}(z_1)$ and
$t_s \in (f^\Tt)^{-1}(z_s)$ such that $\Tt \models t'_1 R_{\lambda_d} t_s$. Let $v'_1$ and $v_s$ be such that
$\a(t'_1) = (y_1, v'_1)$ and $\a(t_s) = (y_s, v_s)$.
Arguing like in (b), one can show that $\tau(v_s) = \tau(v'_s)$ and that $\tau(v_1) = \tau(v'_1)$. But 
$\Tt \models t'_1 R_{\lambda_d} t_s$ together with  the facts that $\a$ is a  homomorphism
 and that $\tau$ is a colouring of $G$ imply that $\tau(v_1) \neq \tau(v'_s)$
(recall that $y_1=x_d$ and  $y_s = x_d'$), a contradiction.
\end{proof}

\section{Pseudoproducts with complete graphs}

Fix a diagram $\D$. For an ordinal $\alpha $ let 
$K_\alpha $ denote the complete graph with $\alpha$ vertices.

\begin{lemma}\label{complete-1}
For any $\alpha$ $\F^\D_\pm \times K_\alpha  \not \models \e^\D(w_0)$.
\end{lemma}

\begin{proof}
For finite $\alpha$ this is a consequence of Lemma~\ref{colouring}, (C2) and the fact that
$\F\models \e^\D(x)$ implies $\F\models \gamma^\D_i$ for all $i \in \omega$ and every Kripke frame $\F$ (Theorem~\ref{Theorem:Axiomatisation}, Soundness). Then notice that
if $\F^\D_\pm \times K_\alpha \models \e^\D(w_0)$ for infinite $\alpha$, then 
$\F^\D_\pm \times K_{\alpha'} \models \e^\D(w_0)$ for some finite $\alpha'$, 
because of the form of $\e^\D(x_0)$.
\end{proof}

For a point $z \in W^{\D,K_\alpha}$ let $\pi_z$ denote the pricipal ultrafilter corresponding to the point $z$.

\begin{lemma}\label{complete-2}
$(\F^\D_\pm \times K_\alpha)^{u.e.} \models \e^\D(\pi_{w_0})$ for every infinite $\alpha$.
\end{lemma}

\begin{proof}
Suppose that $W^\D_\pm = \{w_0,w_1, \ldots, w_{b-1}\}$, . We put 
$W_0 = \{w_0\}$, and $W_i = \{w_i\} \times \alpha$. 
Let $h : W^{\D,K_\alpha}\setminus\{w_0\} \to \alpha$ be the projection given by the formula
$h((w,v)) = v$. 

First, we prove that $(\F^\D_\pm \times K_\alpha)^{u.e.} \models \e^\D(w_0)$. 
To this end we fix an arbitrary non-principal ultrafilter $u$ over $\alpha$, put $X_i = \{g(x_i)\} \times \alpha$ (thus every $X_i = W_j$ for some j) 
and  for $1\le i \le n$ define ultrafilters $\mu_i$ by  the following condition
$$ A \in \mu_i \iffa h (A \cap X_i) \in u.$$
Also, put $\mu_0 = \pi_{w_0}$. It is easy to check that $x_i R^\D_\lambda x_j$ implies
$\mu_i (R^{\D,K_\alpha})^{u.e.}_\lambda \mu_j$,  and so\\  $(\F^\D_\pm \times K_\alpha)^{u.e.} 
\models \k^\D(\mu_0,\mu_1, \ldots, \mu_n)$ (for details, see Section 5 of \cite{KZ}).

Now let us show that for arbitrary ultrafilter $v$, such that $\{w_0\} \notin v$
$(\F^\D_\pm \times K_\alpha)^{u.e.} \models \e^\D(v)$. Notice that $W^{\D,K_\alpha} = \{w_0\}\cup W_1 \cup \ldots\cup W_m.$ 
Hence, by Lemma~\ref{ultra-n} for some $s$ $W_{s} \in v$. 
Let $u$ be  the ultrafilter on $\alpha$ defined by 
condition $$A \in u \iffa h (A \cap W_s) \in u.$$ But, according to  ({\it C-vi}), there
exist 
$w_{k_1}, \ldots, w_{k_n}$ such that $\F^\D_+ \models \k^\D(p(w_{s}), w_{k_1}, \ldots, w_{k_n})$.
Now, define ultrafilters $\mu_i$ for $i = 1, \dots, n$ by the condition 
$$ A \in \mu_i \iffa h (A \cap W_{k_i}) \in u.$$ We claim that 
$(\F^\D_\pm \times K_\alpha)^{u.e.} \models \k^\D(v, \mu_1, \ldots, \mu_n)$. Thus 
$(\F^\D_\pm \times K_\alpha)^{u.e.} \models \forall x (x\neq w_0 \to \e^\D(x))$, and so
$(\F^\D_\pm \times K_\alpha)^{u.e.} \models \e^\D(w_0)$.
\end{proof}

\begin{lemma} \label{complete-3}
Let $u$ be an ultrafilter over $\omega$. Then
$\prod^u_{i\in \omega} (\F^\D_\pm \times K_i)$ is isomorphic to  $\F^\D_\pm \times \prod^u_{i\in \omega} K_i$.
\end{lemma}
\begin{proof}
Let $\lceil z_0, v_0, z_1, v_1, z_2, v_2, \dots \rceil = \lceil \bar z, \bar v \rceil$ be a point of 
$\prod^u_{i\in \omega} (\F^\D_\pm \times K_i)$. 
Set $W_j = \{i \in \omega \mid  z_i = w_j\}$ for $j = 0, \dots, m$. By Lemma~\ref{ultra-n}, there exists
unique $j$ such that $W_j \in u$. Suppose, $W_j = \{i_0, i_1, i_2, \dots\}.$ We put
$f (\lceil z_0, v_0, z_1, v_1, z_2, v_2, \dots \rceil) = (w_j, \lceil x_{i_0}, x_{i_1}, x_{i_2}, \dots \rceil).$
We claim that $f$ is an isomorphism between 
$\prod^u_{i\in \omega} (\F^\D_\pm \times K_i)$ and $\F^\D_\pm \times \prod^u_{i\in \omega} K_i$.
\end{proof}

\section{Erd\"os graphs, or putting it all together}

In this section we finally prove the following theorem, the strongest result of this paper.

\begin{teo}\label{main_can}
Let $\D$ be a minimal connected diagram with inner cycle and let\\
$L = \Log(\forall x_0 \e^\D(x_0))$. Then any axiomatisation of $L$ requires 
infinitely many non-canonical formulas.
\end{teo}

In order to do it, we use the following theorem by I. Hodkinson and Y. Venema. Its proof uses probabilistic
graphs of Paul Erd\"os, and we do not reproduce it.

\begin{teo}[Theorem 2.3 from  \cite{Hodkinson03canonicalvarieties}]\label{t17} 
Let $s \ge 2$. There are finite graphs $H_0, H_1, \ldots$ and surjective homomorphisms 
$\rho_i: H_{i+1} \to H_i$ for $i <\omega$ such that for each $i$,
\begin{enumerate}[({E}1)]
\item for each edge $\{x,y\}$ of $H_i$ and each $x' \in \rho_i^{-1} (x)$, there is $y' \in \rho_i^{-1}(y)$ such that
$\{x',y'\}$ is an edge of $H_{i+1}$,
\item $H_i$ has no odd cycles of length $\le i$,
\item $\chi(H_i) = s$ ($\chi$ is the chromatic number).
\end{enumerate}
\end {teo}

Relying upon this theorem we show that the condition of Lemma \ref{l16} indeed holds for axiomatisation
$\gamma^\D_i$ of the logic in question. Recall that $b$ is the number of points in $\F^\D_{\pm}$, and that
$W^\D_\pm = \{w_0, w_1, \dots, w_{b-1}\}$.

Given $l$, we announce $n = (2^{bl}+1) \cdot (b-1) + 1$. Then, given $k$, we apply Theorem \ref{t17} with 
$s = 2^{bk} + 1$, 
and get a sequence of graphs $H_i$ and surjective homomorphisms
$\rho_i: H_{i+1} \to H_i$. Now, we define the sequence $G_i$ to be the disjoint union of $H_i$ and
$K_{2^{bl}+1}$ (here $K_m$ is the full graph on $m$ vertices), and extend $\rho_i$ to $G_{i+1}$ by putting it
identical on $K_{2^{bl}+1}$.  Finally, we set $\F_i = \F^\D_\pm \times G_i$, and define
morphisms $f_i: \F_{i+1} \to \F_{i}$ by $f_i(w_0) = w_0$ and  
$f_i((w_j, v)) = (w_j, \rho_i(v))$ for $j \ge 1$. (E1) guarantees that all $f_i$ are indeed p-morphisms. 
It is easy to see that
\begin{equation}\label{e15}
\liminv (\F^\D_\pm \times G_i) = \F^\D_{\pm}\times \liminv G_i \quad\mbox{and}
\end{equation}
\begin{equation}\label{e14}
\liminv (G_i) = (\liminv H_i) \cup K_{2^{bl}+1}.
\end{equation}

Now we have apply Lemma \ref{colouring} to ensure  that (L1), (L2), (L3) hold for formulas $\gamma^\D_i$.

(L1): By (E3), $H_i$ has chromatic number $2^{bk} + 1$, and so 
it cannot be painted in $2^{bk}$ colours. Since $H_i$ is a subgraph of $G_i$, $G_i$ also cannot
be painted in $2^{bk}$ colours. Thus, by (C1), $\F_i \models \gamma^\D_k$.

(L2): By (\ref{e14}), $K_{2^{bl}+1}$ is a subgraph of $\liminv G_i$, and so 
$\liminv G_i$ cannot be coloured in $2^{bl}$ colours. Again, by (\ref{e15}) and (C1), 
$\liminv \F_i \models \gamma^\D_l$.

(L3): By (E2), $\liminv H_i$ is two colourable. Hence
$\liminv (G_i)$ can be coloured in $2^{bl}+1$ colours, therefore,
by (C2), $\liminv \F_i \not\models \gamma^\D_n$ for $n  = (2^{bl}+1)\cdot (b-1) + 1$.
This finishes the proof of Theorem~\ref{main_can}.

\section{Main results}

\begin{teo}
Let $\D$ be a rooted  diagram, all undirected cycles of which pass through its root. 
Then ({\it I-i}) -- ({\it I-x}) hold.
\end{teo}
\begin{proof}
It is enough to establish that $\e^\D(x_0)$ is locally modally definable by a generalised Sahlqvist formula, and
then use the generalised Sahlqvist theorem  \cite{Gor2} on completeness. A rigourous proof of modal definability
of $\e^\D(x_0)$ by a generalised Sahlqvist formula can be found in \cite{KZ}, Theorem 4.3. 
Since the proof is quite long, we do not reproduce it here. A shorter proof of modal definability of $\e^\D(x_0)$ was given earlier in \cite{Zolin1}.
\end{proof}

\begin{teo}
Let $\D$ be a minimal rooted diagram with a  cycle not passing through its root. 
Then ({\it I-i}) -- ({\it I-ix}) do not hold for the formula $\e(x_0) = \e^\D(x_0)$ and
$\C$, the class of all Kripke frames validating $\forall x_0 \e(x_0)$.
\end{teo}

({\it I-i}) -- ({\it I-iii}): Since local modal definability implies global definability, 
it is enough to show that  $\forall x_0 \e(x_0)$ is not globally modally 
definable. Indeed, $\F^\D_\pm \times K_\omega \not \models \forall x_0 \e(x_0)$ (Lemma~\ref{complete-1})
but $(\F^\D_\pm \times K_\omega)^{u.e.} \models \forall x_0 \e(x_0)$ (Lemma~\ref{complete-2}), a contradiction
to Lemma~\ref{Lemma:Standard:UE}.

({\it I-iv}), ({\it I-v}), ({\it I-vii}), ({\it I-viii}): follow immediately from Theorem~\ref{main_can}.

({\it I-vi}):  
 As it is known (e.g. from \cite{Agi2}, but the idea dates back to \cite{MSS}), 
to prove that a normal modal logic $L$
is not axiomatisable with finitely many variables, it is sufficient to construct
a sequence of Kripke frames $\F_1, \F_2, \F_3, \dots$, such that 

(a) $\F_i \not \models L$ for all $i$.

(b) for all $k$ there exists $n$ such that $(\F_n, \theta) \models L$ for every 
$k$-generated valuation $\theta$ on $\F_n$.

And this can be easily done: take $\F_i = \F^\D_\pm \times K_i$ for all $i \in \omega$. Then 
(a) follows from Lemma~\ref{colouring}, (C2) and (b) follows from Lemma~\ref{colouring}, (C1), and the
fact that $K_n$ cannot be painted in less then $n$ colours (take $n = 2^{bk} + 1$).

({\it I-ix}): it is clear that $\F^\D_\pm \times K_\omega \notin \C$ (Lemma~\ref{complete-1}), but
 $\F^\D_\pm \times K_\omega \models L$ (Lemma~\ref{colouring}, (C1)).

({\it I-x}): suppose that there exists a first-order formula $\gamma$ such that 
$\F \models \gamma$ iff $\F \models L$ for each Kripke frame $\F$. Then by Lemma~\ref{colouring}, 
(C2) for all $i < \omega$ $\F^\D_\pm \times K_i \models \neg \gamma$, but by Lemma~\ref{complete-3}, 
$\prod^u_{i\in \omega} (\F^\D_\pm \times K_i) = \F^\D_\pm \times (\prod^u_{i\in \omega} K_i)$,
therefore  $\prod^u_{i\in \omega} (\F^\D_\pm \times K_i) \models \gamma$,
since $\prod^u_{i\in \omega} K_i$ is isomorphic to $K_\alpha$ for some infinite $\alpha$ (cf. the proof of 
Theorem~10 of \cite{hughes}). This contradicts Proposition~\ref{ultraproduct}.

\section{Discussion}

Let us discuss the family of Kripke frame classes
that are covered by our theorem. They are defined by first-order formulas of the form
$\forall x_0 \exists x_1 \dots \exists x_n \bigwedge x_i R_\lambda x_j$. 
This family is chosen because it is large enough to generate modal logics of both types of the dichotomy,
and narrow enough to allow the dichotomy to be proven. How interesting is this family?
On the one hand,  these formulas may seem rather artificial, 
since very few of them may be said to be orthodox in modal logic, though
they include well known reflexivity and reflexive-successor conditions. It also seems difficult to invent a practical reasoning problem involving these formulas. On the other hand, if we omit the universal quantifier $\forall x_0$, 
then we obtain  existential conjunctive formulas, which recently have received much attention both in the logical and computer science communities under the name of conjunctive queries.  If we close 
the class of existential conjunctive formulas with many free variables under 
restricted universal quantification, we obtain the class of $\forall\exists$-formulas discussed in Section 6 of \cite{KZ}, which includes many more first-order conditions traditional to  modal logic. 
Thus the formulas from this paper may be understood as `building blocks' for more complicated and interesting formulas, and so our result can be considered as a step towards more general dichotomy theorems. 
First-order formulas of the form $\mathsf{a}(x) = \exists y (x R_\lambda y \land \mathsf{b}(y))$  where $\mathsf{b}(y)$ is a generalised Kracht formula may be good candidates for further research; see \cite{BSS} and \cite{GoldblattHodkinson2006} for known information about the corresponding modal logics. However, it is still not clear how far this dichotomy can be pushed.
It is also interesting if the condition
\begin{enumerate}[{(\it I-xi})]
\item $\{\F \mid \F \models \Log(\C)\}$ is $\Delta$-elementary
 \end{enumerate}
may be added to ({\it I-i}) -- ({\it I-x}) without breaking the dichotomy (cf. \cite{Benthem76a}).

\paragraph{Acknowledgements.} 
The author thanks Philippe Balbiani and Ian Hodkinson for helpful and fruitful discussions. 
This research was supported by  RFBR - CNRS grant  \mbox{11-01-93107}.  The preparation of 
the final version of the paper was supported by RFBR - CNRS grant \mbox{14-01-93105}.


\end{document}